\documentclass{article}
\usepackage{graphicx} 

\usepackage{tikz}
\usetikzlibrary{matrix}
\usetikzlibrary{arrows}
\usepackage{enumerate}
\usepackage{amsmath}
\usepackage{amssymb,latexsym}
\usepackage{amsthm}
\usepackage{tipa}
\usepackage{anysize}
\usepackage{graphicx}
\usepackage{url}
\usepackage{color}
\usepackage{cite}
\usepackage{float}
\usepackage{latexsym}
\usepackage{amsthm}
\usepackage{amssymb}
\usepackage[utf8]{inputenc}
\usepackage{mathrsfs}
\usepackage[english]{babel}
\usepackage{setspace}
\usepackage{listings}
\usepackage{multirow}
\usepackage{algorithm}
\usepackage{framed}
\usepackage{algpseudocode}
\definecolor{citec}{HTML}{324FDA} 
\definecolor{linkc}{HTML}{A0111A}
\definecolor{urlc}{HTML}{b55c87}
\usepackage[
    colorlinks,
    citecolor=citec,
    linkcolor=linkc,
    urlcolor=urlc,
    bookmarks = true,
    breaklinks,
]{hyperref}
\usepackage{cleveref}

\usepackage[margin=2.6cm]{geometry}
\usepackage{algorithm}
\usepackage{algpseudocode}

\newcommand{\RS}{{\mathrm{RS}}}
\newcommand{\cM}{{\mathcal M}}
\newcommand{\cC}{{\mathcal C}}
\newcommand{\cD}{{\mathcal D}}

\newcommand{\cG}{{\mathcal G}}

\newcommand{\F}{{\mathbb F}}

\newcommand{\rk}{{\operatorname{rk}}}

\newcommand{\aaa}{\underline{\alpha}}

\newcommand{\la}{\langle}
\newcommand{\ra}{\rangle}

\newcommand\qbin[3]{\left[\begin{matrix} #1 \\ #2 \end{matrix}\right]_{#3}}

\newcommand{\Fq}{\mathbb{F}_q}

\newcommand{\rc}[1]{{\color{blue}[Roni: #1]}}

\theoremstyle{definition}

\newtheorem{theorem}{Theorem}[section]
\newtheorem{lemma}[theorem]{Lemma}
\newtheorem{corollary}[theorem]{Corollary}
\newtheorem{definition}[theorem]{Definition}
\newtheorem{proposition}[theorem]{Proposition}

\newtheorem{question}[theorem]{Question}

\newtheorem{remark}[theorem]{Remark}

\newtheorem{claim}[theorem]{Claim}

\title{Sunflowers of Reed--Solomon Codes}

\usepackage{authblk}

\author[1]{Roni Con}
\affil[1]{Tel Aviv University, Israel}

\author[2]{Anina Gruica}
\affil[2]{Technical University of Denmark, Denmark}

\author[2]{Maria Montanucci}

\author[3]{Ferdinando Zullo}
\affil[3]{Università degli Studi della Campania ``Luigi Vanvitelli'', Italy}

\date{}

\begin{document}

\maketitle

\begin{abstract}
We introduce and study \emph{Reed--Solomon sunflowers}, namely families of Reed--Solomon codes whose pairwise intersections are all equal to the same fixed subspace. This notion lies at the intersection of extremal subspace combinatorics and coding theory: it can be viewed as a structured version of the sunflower problem in the Grassmannian, and it naturally produces constant-dimension subspace codes with prescribed minimum distance.
We focus mainly on the case in which the center is the one-dimensional space generated by the all-one vector. We give an algebraic criterion, expressed in terms of generalized $V$-matrices, ensuring that a family of Reed--Solomon codes forms such a sunflower. We then study the size of these families through counting and constructions. In dimension two, we show that all distinct Reed--Solomon codes form a sunflower and determine its size by counting Reed--Solomon codes up to affine equivalence of their evaluation vectors.
For fixed dimension $k\geq3$ and length $\ell\geq2k-1$, we give an explicit recursive construction with $\Omega_{k,\ell}(q^{\lfloor\ell/(2k-1)\rfloor})$ petals and a greedy existence argument with $\Omega_{k,\ell}(q^{\ell-2k+2})$ petals as $q\to\infty$.
We also apply the greedy argument to obtain families of $[\ell,k]_q$ MDS codes of size $\Omega_{k,\ell}(q^{2(\ell-2k+2)})$, whose pairwise intersections have dimension at most one but need not be equal.
\end{abstract}



\newpage
\tableofcontents
\newpage

\section{Introduction}

The study of large families of sets with prescribed intersection patterns is a central theme in extremal combinatorics. A classical starting point is the Erdős--Ko--Rado theorem, which determines the largest size of an intersecting family of $k$-subsets of an $n$-set, under suitable assumptions on the parameters \cite{erdos1961intersection}. This result initiated a broad line of research on extremal set systems, including $t$-intersecting families, complete intersection theorems, and related stability questions; see, for instance, \cite{ahlswede1997complete,wilson1984intersecting}.

A closely related notion is that of a sunflower, or $\Delta$-system, introduced by Erdős and Rado \cite{erdos1960intersection}. A sunflower is a family of sets whose pairwise intersections are all equal to the same fixed set, called the center. Sunflowers may be viewed as highly structured intersecting families: while the Erdős--Ko--Rado problem asks for large families in which all pairwise intersections are nonempty, the sunflower condition prescribes the intersection itself. This stronger requirement has led to a rich theory with connections to extremal combinatorics, Ramsey theory, and theoretical computer science.

A natural linear analogue of these questions arises when one replaces subsets by subspaces of a finite-dimensional vector space. Let $\mathcal{G}_q(k,n)$ denote the Grassmannian of all $k$-dimensional subspaces of $\mathbb{F}_q^n$. A family $\mathcal{C}\subseteq \mathcal{G}_q(k,n)$ is called intersecting if any two of its elements have nontrivial intersection. More generally, one can ask for families in which the intersection of any two members has dimension at least $t$. This is the $q$-analogue of the Erdős--Ko--Rado problem and has been widely studied; see, for example, \cite{hsieh1975intersection,frankl1986erdos,blokhuis2022sunflower,otal2026set,ihringer2026erdHos,bartoli2021improvement}. In this setting, extremal families are often given by all $k$-subspaces containing a fixed $t$-dimensional subspace, which is the vector-space analogue of a star.

Sunflowers in the Grassmannian form a particularly important class of such families. A subspace sunflower is a family of $k$-dimensional subspaces whose pairwise intersections are all equal to the same fixed subspace. In other words, a sunflower is not only $t$-intersecting, but has a prescribed common intersection. These objects have also appeared naturally in the theory of subspace codes. Indeed, when subspaces are used as codewords and the distance is the subspace distance
\[
d(U,V)=2\bigl(k-\dim_{\mathbb{F}_q}(U\cap V)\bigr),
\]
the dimension of the pairwise intersection directly determines the distance between codewords. Thus, a sunflower with center of dimension $c$ gives a constant-dimension subspace code with minimum distance $2(k-c)$. For this reason, sunflowers have been investigated as structured examples of equidistant or nearly equidistant subspace codes, and their relation with partial spreads, designs, and extremal configurations has been studied in several works; see, for instance, \cite{etzion2011error,etzion2015equidistant}.

The general problem of constructing large subspace codes with prescribed intersection behavior is motivated in part by random linear network coding. In the noncoherent model introduced by Koetter and Kschischang \cite{koetter2008coding}, transmitted information is represented by subspaces rather than by individual vectors. Constant-dimension codes in the Grassmannian therefore play the role of error-correcting codes for the operator channel. From this point of view, a sunflower is an especially transparent configuration: all codewords have the same dimension, and all pairwise distances are controlled by the dimension of the common center. However, arbitrary subspace sunflowers do not necessarily carry additional algebraic structure that can be exploited for encoding, recognition, or decoding.

In this paper we study sunflowers whose petals are \emph{Reed--Solomon codes} (see \Cref{def:RS-code}). Specifically, we consider families
\[
\mathcal{S}=\{\mathsf{RS}_{\ell,k}(\underline{\alpha}_1),\ldots,
\mathsf{RS}_{\ell,k}(\underline{\alpha}_s)\}
\]
of $k$-dimensional Reed--Solomon codes (RS codes, for short) in $\mathbb{F}_q^\ell$, with the property that every pair of distinct codes in the family has the same fixed subspace as their intersection. We call such a family a \emph{Reed--Solomon sunflower} (RS sunflower, for short). The main case considered in this work is the one-dimensional center
\[
\langle \underline{1}\rangle_{\mathbb{F}_q},
\qquad 
\underline{1}=(1,\ldots,1),
\]
which is naturally contained in every RS code, since it is obtained by evaluating constant polynomials. Thus we are interested in families satisfying
\[
\mathsf{RS}_{\ell,k}(\underline{\alpha}_i)
\cap
\mathsf{RS}_{\ell,k}(\underline{\alpha}_j)
=
\langle \underline{1}\rangle_{\mathbb{F}_q}
\qquad
\text{for all } i\ne j.
\]
Such a family is simultaneously a structured object in extremal subspace combinatorics and a constant-dimension subspace code with minimum distance $2k-2$.

The additional Reed--Solomon structure distinguishes our problem from the general theory of subspace sunflowers. From the combinatorial point of view, we are not asking merely for large sunflowers in the Grassmannian, but for large sunflowers whose petals belong to a very special algebraic class of MDS (maximum distance separable) codes. From the coding-theoretic point of view, this restriction is useful because RS codes admit strong algebraic tools, including explicit descriptions by evaluation points, efficient encoding and decoding algorithms, and characteristic behavior under the Schur product. These features make RS sunflowers suitable not only for existence and counting questions, but also for algorithmic applications. 

More broadly, RS codes are among the most widely
used families of error-correcting codes in both theory and practice, with numerous applications across diverse domains. Examples include distributed storage systems, QR codes, secret sharing and cryptography more generally, computational complexity, and others.  
We expect that RS sunflowers may also prove useful beyond the immediate application in subspace codes.

Our first contribution is an algebraic criterion ensuring that a family of RS codes forms a sunflower with center $\langle\underline{1}\rangle_{\mathbb{F}_q}$. The criterion is formulated in terms of a generalized $V$-matrix associated with two evaluation vectors. Full rank of this matrix rules out nonconstant common codewords and therefore guarantees that the two corresponding RS codes intersect only in the space of constant vectors. This gives a concrete way to translate the sunflower condition into rank conditions on evaluation points.

We then study the size of RS sunflowers through counting and constructions. In dimension $k=2$, the situation is especially simple: any two distinct two-dimensional RS codes already intersect exactly in $\langle\underline{1}\rangle_{\mathbb{F}_q}$. We therefore obtain the exact size of the corresponding sunflower by counting RS codes up to affine equivalence of their evaluation vectors. For fixed $k\geq3$ and $\ell\geq2k-1$, our explicit recursive construction, based on the generalized $V$-matrix criterion, yields $\Omega_{k,\ell}(q^{\lfloor\ell/(2k-1)\rfloor})$ petals as $q\to\infty$.

We also obtain a greedy existence bound. Specifically, we show that a fixed $[\ell,k]_q$ RS code intersects at most $2(k-1)^{\ell-2}q^{2k-4}$ RS codes in dimension at least two. Starting with all distinct RS codes, we repeatedly select a code and discard all codes whose intersection with it has dimension at least two. Combining the resulting bound with our exact count of RS codes yields RS sunflowers with $\Omega_{k,\ell}(q^{\ell-2k+2})$ petals.

For $k\geq3$, the recursive and greedy RS bounds have the same exponent when $\ell=2k-1$, and the greedy bound has the larger exponent when $\ell\geq2k$. Applying the same selection argument to MDS codes gives families of size $\Omega_{k,\ell}(q^{2(\ell-2k+2)})$ with minimum subspace distance at least $2k-2$. These families need not share a common center.

The paper is organized as follows. In Section~2 we introduce RS sunflowers and prove the generalized $V$-matrix criterion. In Section~3 we count RS codes, discuss the case $k=2$, show a connection between RS codes that can correct deletions and RS sunflowers, and give explicit recursive constructions for general $k$. We also derive greedy existence bounds for MDS code families and RS sunflowers and compare them with the recursive construction. In Section~4, we discuss applications to random linear network coding and explain how the Reed--Solomon structure can be exploited algorithmically through Schur-product distinguishers and reconstruction methods. Finally, in Section~5, we conclude with open questions for future research.

\section{Reed--Solomon sunflowers}

In this section, we introduce the main object of the paper. We first give the relevant definitions and then formulate the main question studied here. Throughout, for a prime power $q$, we denote by $\F_q$ the finite field of $q$ elements, and we let $1 \le c < k \le \ell$ be integers.
Throughout the paper, for a positive integer $m$, we write $[m]:=\{1,\ldots,m\}$.

\begin{definition}[Reed--Solomon code] \label{def:RS-code}
    Let $\alpha_1,\ldots,\alpha_n \in \Fq$ be distinct points in a finite field $\Fq$ of order $q\geq n$.
    For $k\leq n$, the $[n,k]_q$ \textbf{Reed--Solomon code} defined by the evaluation vector $\underline{\alpha}=(\alpha_1,\ldots,\alpha_n)$ is the set of codewords
    \[ \mathsf{RS}_{n,k}(\underline{\alpha})=\{ (f(\alpha_1),\ldots,f(\alpha_n)) \colon f \in \Fq[x]_{<k} \}, \]
    where $\Fq[x]_{<k}$ denotes the set of polynomials in $\Fq[x]$ with degree smaller than $k$.
\end{definition}

We are interested in families of RS codes which, when viewed as subspaces of $\mathbb{F}_q^\ell$, form a \textit{sunflower}: all pairwise intersections are equal to the same fixed subspace.  This combines two types of structural properties.  On the one hand, RS codes have strong algebraic properties and admit efficient recognition and decoding algorithms.  On the other hand, the sunflower condition gives a \textit{constant-dimension subspace} code with controlled subspace distance.

\begin{definition}[Reed--Solomon sunflower]
Let $1\le c < k \le \ell \le q$, and let 
\[
\mathcal{S}=\{\RS_{\ell, k}(\underline{\alpha}_1),\ldots,\RS_{\ell, k}(\underline{\alpha}_s)\}
\]
be a family of $k$-dimensional RS codes in $\mathbb{F}_q^\ell$.
We say that $\mathcal{S}$ is an $(\ell,k,q;c)$ \textbf{Reed--Solomon sunflower} if there exists a fixed
$c$-dimensional subspace $E\subseteq \mathbb{F}_q^\ell$ such that
\[
\RS_{\ell, k}(\underline{\alpha}_i)\cap \RS_{\ell, k}(\underline{\alpha}_j)=E
\qquad \text{for all } i\ne j.
\]
The subspace $E$ is called the \emph{center} of the sunflower, and the codes
$\RS_{\ell, k}(\underline{\alpha}_i)$ are called its \emph{petals}.
\end{definition}

The set of all $k$-dimensional $\mathbb{F}_q$-subspaces of $\mathbb{F}_q^n$ is called the \textbf{Grassmannian} over $\mathbb{F}_q$ and is denoted by $\mathcal{G}_q(k,n)$. An $(\ell,k,q;c)$ RS sunflower is, in particular, a family of $k$-dimensional subspaces of $\mathbb{F}_q^\ell$. Hence it can also be viewed as a constant-dimension subspace code in $\mathcal{G}_q(k,\ell)$. Under the subspace distance
$$d(U,V)=2\bigl(k-\dim_{\mathbb{F}_q}(U\cap V)\bigr),$$
a sunflower with center dimension $c$ has minimum distance $2(k-c)$. Thus, in the case $c=1$, the sunflowers studied here give constant-dimension codes with minimum distance $2k-2$.

\begin{remark} \label{rem:sunflower-subspace-codes}
Sunflower codes have been extensively studied in the context of
constant-dimension codes.  If \(\mathcal S\subseteq \mathcal{G}_q(k,n)\) is a
\(t\)-intersecting sunflower with center \(\textrm{Cen}(\mathcal S)\), then, after
quotienting by \(\textrm{Cen}(\mathcal S)\), the petals form a \emph{partial
\((k-t)\)-spread} in a complementary \((n-t)\)-space, i.e. a family of \((k-t)\)-dimensional subspaces pairwise intersecting trivially in an  \((n-t)\)-space 
Conversely, every
partial \((k-t)\)-spread in \(\mathcal{G}_q(k-t,n-t)\) gives rise, by adding a fixed
\(t\)-dimensional center, to a \(t\)-intersecting sunflower in \(\mathcal{G}_q(k,n)\).
Thus, the construction of large sunflower codes is essentially equivalent to
the construction of large partial spreads.
Several constructions of sunflower codes therefore arise from known
constructions of partial spreads.  Classical spreads give optimal sunflower
codes whenever \(k-t\) divides \(n-t\) of size 
\[
\frac{q^{n-t}-1}{q^{k-t}-1},
\]
see \cite{segre1964teoria}.  For the remaining parameters, general
constructions go back to Beutelspacher~\cite{Beutelspacher1975}, and further
families can be obtained from lifted rank-metric codes, linkage constructions,
and echelon--Ferrers constructions; see, for instance,
\cite{koetter2008coding,etzion2011error}.
Systematic constructions and structural results for equidistant subspace
codes, including sunflower codes and their orthogonals, were later studied in
\cite{etzion2015equidistant,gorla2014partial}.  In particular, Etzion and Raviv \cite{etzion2015equidistant}
showed that, for large ambient dimension, the largest equidistant
constant-dimension codes are sunflowers, and that the size of the largest
sunflower is governed by the corresponding partial-spread problem.
\end{remark}

The basic problem we study is therefore the following: for fixed parameters $q,\ell,k$ and a prescribed center dimension $c$, how large can a family of $k$-dimensional RS codes in $\mathbb{F}_q^\ell$ be if all pairwise intersections are equal to the same $c$-dimensional subspace?  For the rest of this section we concentrate on the case $c=1$ and center $\langle(1,\ldots,1)\rangle_{\Fq}$. From now on, we denote by $\underline{1}$ the vector $(1,\dots,1)$. 
This center is naturally contained in every RS code, since it is obtained by evaluating constant polynomials. Formally, we are interested in the following question.

\begin{question} \label{question}
    For fixed parameters $q,\ell$, and $k$, what is the largest possible size of an $(\ell,k,q;1)$ RS sunflower? Furthermore, is it possible to explicitly construct sunflowers of this type whose size is (asymptotically) large?
\end{question}

Our main ingredient in tackling Question~\ref{question} is the following concrete algebraic criterion ensuring that two RS codes intersect only in this center. For this, we use an extension of the $V$-matrix considered in \cite[Section 2.1]{con2023reed}. 

\begin{definition}[Generalized $V$-matrix]\label{matrix} 
    For positive integers $\ell, k$ define the $\ell\times(2k-1)$ matrix
\[      V_{k,\ell}(\underline{X}, \underline{Y})=   V_{k,\ell}(X_1,X_2,\ldots,X_{\ell},Y_1,Y_2,\ldots,Y_{\ell})=\left(\begin{array}{ccccccc}1 & X_{1} & \cdots & X_{1}^{k-1} & Y_{1} & \cdots & Y_{1}^{k-1} \\1 & X_{2} & \cdots & X_{2}^{k-1} & Y_{2} & \cdots & Y_{2}^{k-1} \\ \vdots& \vdots & \ddots& \vdots & \vdots & \ddots & \vdots \\1 & X_{\ell} & \cdots & X_{\ell}^{k-1} & Y_{\ell} & \cdots & Y_{\ell}^{k-1}\end{array}\right),     \] in 
$\F_q[X_1,\dots,X_{\ell},Y_1,\dots,Y_{\ell}]^{\ell \times (2k-1)}$. We call $V_{k,\ell}(\underline{X}, \underline{Y})$
a \textbf{generalized $V$-matrix}. 
\end{definition}

The generalized $V$-matrix records the equations that would have to be satisfied by two polynomials of degree less than $k$ whose evaluations agree on the two prescribed evaluation vectors.  Full rank therefore rules out nonconstant common codewords and gives the following construction criterion.

\begin{theorem}\label{thm:constr}
Let $k$, $\ell$ and $s$ be positive integers with $\ell \ge 2k-1$. Let $\mathcal{A}=\{\underline{\alpha}_1, \ldots, \underline{\alpha}_s\}$ be a collection such that for every $i \in [s]$, we have $\underline{\alpha}_i = (\alpha_{i,1},\ldots,\alpha_{i,\ell}) \in \mathbb{F}_q^\ell$, and the entries $\alpha_{i,h}$ are pairwise distinct within each vector. \footnote{Note that two different evaluation vectors $\underline{\alpha}_i$ and $\underline{\alpha}_j$ may share some coordinates; the requirement is only that, inside any single evaluation vector, all coordinates are distinct.}
Moreover, assume that for all distinct $i,j \in [s]$, we have $\mathrm{rk}(V_{k,\ell}(\underline{\alpha}_i,\underline{\alpha}_j)) = 2k-1$. Then the set  
 \[
    \mathcal{C}=\lbrace \mathsf{RS}_{\ell,k}(\underline{\alpha}_{i}) \colon  i\in [s] \rbrace
    \]
is an $(\ell,k,q;1)$ RS sunflower. Moreover, for all $i,j\in [s], i\neq j$, we have that $\mathsf{RS}_{\ell,k}(\underline{\alpha}_{i}) \cap \mathsf{RS}_{\ell,k}(\underline{\alpha}_{j})=\langle \underline{1} \rangle_{\Fq}$.
\end{theorem}
\begin{proof}
    Since for every $i\in [s]$, we have that $\underline{\alpha}_i \in \Fq^{\ell}$ contains pairwise distinct points, we have that $\mathsf{RS}_{\ell,k}(\underline{\alpha}_{i}
    )$ is of dimension $k$ and length $\ell$.
    Hence, we are left to determine the intersection between any two elements of $\cC$. 
    To do so, suppose that there exist 
    $i,j\in [s]$ with $i \ne j$ such that the dimension, as $\mathbb{F}_q$-vector space, of the intersection of $\mathsf{RS}_{\ell,k}(\underline{\alpha}_{i})$ and $\mathsf{RS}_{\ell,k}(\underline{\alpha}_{j})$ is at least $2$. Then there exist two polynomials $f(x)=\sum_{i=0}^{k-1}f_ix^i,g(x)=\sum_{h=0}^{k-1} g_h x^h \in \mathbb{F}_q[x]_{<k}$ (not both constant) such that
    $$c=(f(\alpha_{i,1}),\ldots,f(\alpha_{i,\ell}))=(g(\alpha_{j,1}),\ldots,g(\alpha_{j,\ell})),$$
    where $\underline{\alpha}_i=(\alpha_{i,1},\ldots,\alpha_{i,\ell})$ and $\underline{\alpha}_j=(\alpha_{j,1},\ldots,\alpha_{j,\ell})$.
    Define $f' = f - g_0$ and $g' = g - g_0$. Then, since $\underline{1} \in \mathsf{RS}_{\ell,k}(\underline{\alpha}_{i}) \cap \mathsf{RS}_{\ell,k}(\underline{\alpha}_{j})$, we have that $c - (g_0, \ldots, g_0) \in \mathsf{RS}_{\ell,k}(\underline{\alpha}_{i}) \cap \mathsf{RS}_{\ell,k}(\underline{\alpha}_{j})$ and 
    \[
    c - (g_0, \ldots, g_0) = (f'(\alpha_{i,1}), \ldots, f'(\alpha_{i,\ell})) = (g'(\alpha_{j,1}), \ldots, g'(\alpha_{j,\ell}))\;.
    \]
    
    This implies that $v= (f_0 - g_0, f_1, \ldots, f_{k-1}, -g_1, \ldots,-g_{k-1})$ is a nonzero vector such that $V_{k, \ell}(\underline{\alpha}_i,\underline{\alpha}_j)v^T=\underline{0}$.
    Since, $\textrm{rk}(V_{k,\ell}(\underline{\alpha}_i,\underline{\alpha}_j))=2k-1$ we have that $v=\underline{0}$ and so $f=g=g_0 \in \Fq$, in contradiction to $f$ and $g$ being nonconstant. The moreover part follows by noting that $\langle \underline{1} \rangle_{\Fq}$ is contained in any RS code.
\end{proof}


\section{Constructing and counting RS sunflowers}

We begin with a counting result for RS codes, which is particularly useful in the case $k=2$, where distinct RS codes already have the desired pairwise intersection behavior. We then give explicit recursive constructions based on the rank criterion from Theorem~\ref{thm:constr}. Finally, we prove existence bounds by greedy selection, both in the larger class of MDS codes and in the restricted class of RS codes, and compare them with the recursive construction.

\subsection{Counting RS codes and the case $k=2$}
We begin by counting the number of distinct RS codes of a given length and dimension. This serves two purposes. First, when $k=2$, any two distinct RS codes intersect only in the subspace of constant vectors, so this counting result immediately yields RS sunflowers. Second, we will later use the same count to quantify how sparse RS codes are among all subspaces of a fixed dimension.
The approach begins by identifying a criterion, formulated via evaluation points, that determines when two such RS codes coincide. After establishing this characterization, the desired counting result follows directly as a corollary.

\begin{lemma} \label{RS-eq}
Let $k \geq 2$ and $2(k-1)<n \leq q$. Let $\underline{x}=(x_1,\ldots,x_n)$ and $\underline{y}=(y_1,\ldots,y_n)$ be two vectors in $\mathbb{F}_q^n$ both having pairwise distinct coordinates. Then for the corresponding RS codes $\mathsf{RS}_{n,k}(\underline{x})$ and $\mathsf{RS}_{n,k}(\underline{y})$ one has that
$$\mathsf{RS}_{n,k}(\underline{x})=\mathsf{RS}_{n,k}(\underline{y})$$
if and only if there exist $a,b \in \mathbb{F}_q$ with $a \ne 0$ such that $y_i=ax_i+b$ for all $i=1,\ldots,n$. In other words, two evaluation vectors give rise to the same RS code if and only if they are related by an affine transformation.
\end{lemma}
\begin{proof}
We prove one implication at a time.

$(\Leftarrow)$. Let $a,b \in \mathbb{F}_q$ with $a \ne 0$ such that $y_i=ax_i+b$ for all $i\in [n]$. Let $f(x)$ be a polynomial of degree $<k$. It holds that $f(y_i) = f (h(x_i))$ where $h(x_i) = ax_i + b$ and therefore, $g:= f\circ h$ is of degree $< k$. This implies that $\RS_{n,k}(\underline{y}) \subseteq \RS_{n,k}(\underline{x})$. The other direction follows similarly by noting that $x_i = (y_i - b)/a$ is also linear polynomial.

$(\Rightarrow)$. Assume that $\mathsf{RS}_{n,k}(\underline{x})=\mathsf{RS}_{n,k}(\underline{y})$. Then, $\underline{y} \in \RS_{n,k}(\underline{x})$ which implies that there exists a degree $< k$ polynomial $f$ such that $y_i = f(x_i)$ for all $i\in [n]$. Note that $f$ must be non constant as otherwise we would get contradiction to the assumption that $(y_1, \ldots, y_n)$ are pairwise distinct. 
In general, for every $s \in [k-1]$, we have that $(y_1^s, \ldots, y_n^s)\in \RS_{n,k}(\underline{x})$, and thus, there exists a polynomial $g_s$ of degree $<k$ such that for all $i\in [n]$ we have $y_i^s = g_s(x_i)$. Combining, we have 
\[
g_s(x_i) = y_i^s = (f(x_i))^s \;.
\]

This implies that for all $s\in [k-1]$ the polynomial $G_s(x) = g_s(x) - (f(x))^s$ is zero on $x_1, \ldots, x_n$, i.e., has at least $n$ roots. 
Now assume that $\deg(f) \geq 2$ and set $s' = \lfloor(k-1)/\deg(f) \rfloor + 1 \leq k-1$. 
We have that 
\begin{equation}\label{eq:s-prime-lb}
    \deg((f(x))^{s'}) = s' \cdot\deg(f) > k-1
\end{equation}
and on the other hand, we have
\[
\deg((f(x))^{s'}) = s' \cdot\deg(f) \leq k-1 + \deg(f) \leq 2(k-1) < n \;,
\]
where the last inequality is according to our stated assumption about $k$ and $n$. 
Now, since $\deg(g_{s'}) \leq k-1$ we have that $\deg(G_{s'}) < n$, we get that $G_{s'}$ is the zero polynomial. 
This implies that
\[
\deg((f(x))^{s'}) = \deg(g_{s'}(x)) \leq k-1 \;,
\]
which contradicts \eqref{eq:s-prime-lb}. Thus, we get that $f(x) = ax + b$ with $a\neq 0$.
\end{proof}

\begin{remark}
    We note that only the proof of the $(\Rightarrow)$ direction relied on the assumption $2(k-1) < n$, and one can verify that the $(\Leftarrow)$ implication is valid for all $k$.
    Nevertheless, we stated \Cref{RS-eq} in this form because, throughout this paper, we focus exclusively on the setting where $\ell$, the length of the vectors in the ambient space, satisfies $\ell \ge 2k-1 > 2k-2$.
\end{remark}

This Lemma implies the following corollary.
\begin{corollary} \label{RS-count}
 Let $k \geq 2$ and $2(k-1)<n \leq q$. Then, the number of distinct RS codes of dimension $k$ and length $n$ over $\mathbb{F}_q$ is 
 \[
 v(q,n):=\frac{(q-2)!}{(q-n)!} \;.
 \]   
\end{corollary}
\begin{proof}
The number of vectors $\underline{x}\in\mathbb{F}_q^n$ with pairwise distinct coordinates is
$$|\{\underline{x} \in \mathbb{F}_q^n \mid x_i \ne x_j \text{ for } i \ne j\}|=q(q-1)(q-2)\cdots(q-(n-1))=\frac{q!}{(q-n)!}.$$
By Lemma~\ref{RS-eq}, two such vectors define the same RS code if and only if they differ by an affine transformation $T\mapsto aT+b$, with $a\ne 0$. Since each vector has at least two distinct coordinates, this action is free, and each equivalence class has size $q(q-1)$. Therefore, the number of distinct RS codes of dimension $k$ and length $n$ over $\mathbb{F}_q$ is
$$\frac{q!}{q(q-1)(q-n)!}=\frac{(q-2)!}{(q-n)!},$$
proving the statement of the corollary.
\end{proof}

For $k=2$, this count immediately gives a sunflower. Indeed, every two-dimensional RS code contains the all-one vector, and two distinct two-dimensional subspaces containing this vector intersect exactly in its span.

\begin{corollary}
For $2<\ell\leq q$, the set of all distinct RS codes of dimension $2$ and length $\ell$ over $\mathbb{F}_q$ is an $(\ell,2,q;1)$ RS sunflower with center $\langle\underline{1}\rangle_{\Fq}$ and size
$$\frac{(q-2)!}{(q-\ell)!} = \Omega_{\ell}(q^{\ell-2})\;.$$
\end{corollary}

\begin{remark}
A formula similar to that in Corollary \ref{RS-count} was derived for GRS codes in \cite[Theorem 2.1]{BGH}. However, the methods employed there are quite different, and the theorem’s formula cannot be applied to our setting in a straightforward manner.
\end{remark}

\subsection{From codes correcting deletions to sunflowers}
Our first approach to building sunflowers of Reed–Solomon codes begins with the observation that the generalized $V$-matrix introduced in \Cref{matrix} serves as a key object in the study of Reed–Solomon codes correcting \emph{deletions}. A deletion is an error event in which one symbol of the transmitted codeword is entirely removed, so the receiver obtains fewer symbols and consequently loses synchronization.
The behavior and performance of Reed–Solomon codes with deletion-correction functionality have been thoroughly studied in a series of recent works \cite{duc2019explicit,liu20212,con2023reed,liu2024optimal, CGLZ25,beelen2025reed}.

In this section, we show that an algebraic condition on Reed–Solomon codes that is sufficient for deletion-correction naturally induces a collection of subspaces, each of which is an RS code, with the property that any two distinct subspaces intersect exactly in the common one-dimensional subspace $\langle \underline{1} \rangle _{\Fq}$.
In principle, this correspondence could be extended to broader classes of linear deletion-correcting codes (beyond Reed–Solomon codes). 
To our knowledge, no prior work has explicitly identified such a connection between linear deletion-correcting codes and subspace codes.

We call a sequence $I=(I_1, \ldots, I_{\ell}) \in [n]^{\ell}$ an \emph{increasing sequence} if $I_1 < \cdots < I_{\ell}$. In \cite{con2023reed,CGLZ25}, the following definition of a $V_{k, \ell, I,J}$-matrix was given (note that this definition is similar to our definition of a $V$-matrix).

\begin{definition}[$V_{k,\ell,I,J}$-matrix \cite{CGLZ25}]\label{matrix-insdel}
    For positive integers $\ell, k$ and increasing sequences $I=(I_1,\dots,I_\ell),J=(J_1,\dots,J_\ell)\in [n]^{\ell}$ of length $\ell$, define the $\ell\times(2k-1)$ matrix
    \[
        V_{k,\ell,I,J}(X_1,X_2,\cdots,X_n):=\left(\begin{array}{ccccccc}1 & X_{I_1} & \cdots & X_{I_1}^{k-1} & X_{J_1} & \cdots & X_{J_1}^{k-1} \\1 & X_{I_2} & \cdots & X_{I_2}^{k-1} & X_{J_2} & \cdots & X_{J_2}^{k-1} \\ \vdots& \vdots & \ddots& \vdots & \vdots & \ddots & \vdots \\1 & X_{I_\ell} & \cdots & X_{I_\ell}^{k-1} & X_{J_\ell} & \cdots & X_{J_\ell}^{k-1}\end{array}\right),
    \]
over the field 
$\F_q(X_1,\dots,X_n)$. 
\end{definition}

Then, in \cite{con2023reed,CGLZ25}, the authors presented a sufficient algebraic condition for an $\mathsf{RS}_{n,k}(\underline{\alpha})$ code to correct deletions.
\begin{lemma}\cite[Lemma 13]{CGLZ25}\label{lem:corr-cond}
    Let $\ell\ge 2k-1$. Consider the $[n, k]_q$ RS code $\mathsf{RS}_{n,k}(\alpha_1, \ldots, \alpha_n)\subseteq\F_q^n$ associated with an evaluation vector $\underline{\alpha}:=\left(\alpha_1, \ldots, \alpha_n\right)\in\F_q^n$. 
    If the code cannot correct $n-\ell$ deletions, then there exist two increasing sequences $I=(I_1,\dots,I_\ell), J=(J_1,\dots,J_\ell) \in[n]^{\ell}$ that agree on at most $k-1$ coordinates
    such that matrix $V_{k,\ell,I, J}(\underline{\alpha})$ does not have full column rank.
\end{lemma}

We now show that the same condition guarantees the existence of a family of RS codes of dimension $k$ and length $\ell$ whose pairwise intersections are all exactly $\langle \underline{1}\rangle_{\Fq}$.

\begin{proposition} \label{prop:from-del-to-subspaces}
Let $k, \ell, n \geq2$ be natural numbers such that $\ell \geq 2k-1$ and $n > 2\ell^2$. Let $\alpha_1, \ldots, \alpha_n \in \F_q$ be $n$ distinct elements and denote $\underline{\alpha} = (\alpha_1, \ldots, \alpha_n)$. Suppose that for every pair of increasing sequences $I, J \in [n]^{\ell}$ that agree on at most $k-1$ coordinates, we have $\rk (V_{k,\ell,I,J}(\underline{\alpha})) = 2k-1$. Then there exists an $(\ell, k, q;1)$ RS sunflower with at least $p^k$ petals, for some prime $p$ satisfying $p \in [n/2\ell, n/\ell]$.

Moreover, there is an encoding algorithm that given a message $m \in \F_p^k$, outputs a generator matrix of an RS code from this sunflower using $O(k\cdot \ell)$ operations over $\F_p$, $O(\ell)$ arithmetic operations over $\mathbb{Z}$ with all involved integers smaller than $n$, and $O(k\cdot \ell)$ operations over $\F_q$.
\end{proposition}
\begin{proof}
    Let $\underline{\alpha} = (\alpha_1, \ldots, \alpha_n) \in \Fq^n$ be such that the stated condition holds.
    Let $\mathcal{P} \subset [n]^{\ell}$ be a set of increasing sequences of length $\ell$ such that for any two distinct $I,J \in \mathcal{P}$, we have that $I$ and $J$ agree on at most $k-1$ coordinates. 
    Consider the following set of RS codes
    \[
    \lbrace \mathsf{RS}_{\ell,k}(\underline{\alpha}_{I}) \mid I \in \mathcal{P} \rbrace \;.
    \]
    Clearly, each RS code has length $\ell$ and dimension $k$. We first show that any such family $\mathcal{P}$ yields the desired sunflower. We then construct a suitable $\mathcal{P}$ of size $p^k$.
    For distinct $I,J\in\mathcal{P}$, our assumption gives
    \[
    \rk V_{k,\ell}(\underline{\alpha}_I,\underline{\alpha}_J)
    =\rk V_{k,\ell,I,J}(\underline{\alpha})=2k-1.
    \]
    Thus, \Cref{thm:constr} shows that the corresponding RS codes form a sunflower with $|\mathcal{P}|$ petals.

    We now describe how to construct a set $\mathcal{P}$ such that any two distinct elements $I, J \in \mathcal{P}$ agree in at most $k-1$ positions.
    Let $\cC$ be an $[\ell, k]_{p}$ MDS code, where $p$ is a prime satisfying $n/2\ell \leq p \leq n/\ell$ (such a prime exists by Bertrand’s postulate). Note here that by our assumption on $n$ and $\ell$ given in the statement, we have $p \geq n/2\ell > \ell$, and thus, one can take $\cC$ to be an $[\ell, k]_p$ RS code.
    For each $i \in [\ell]$, we define a function
\begin{align*}
    \varphi_i: \mathbb{F}_p &\to [n]\\
    \beta &\mapsto (i-1)p + \beta + 1,
\end{align*}
where the arithmetic is carried out in $\mathbb{Z}$, identifying elements of $\mathbb{F}_p$ with the integers in $\{0,1,\ldots,p-1\}$. Clearly, each $\varphi_i$ is a bijection from $\F_p$ onto $\{(i-1)p + 1, \ldots,ip\}$.

    Now define
    \[
    \mathcal{P} = \left\{ (\varphi_1(c_1), \ldots, \varphi_{\ell}(c_{\ell})) \mid \underline{c} \in \cC \right\} \;.
    \]
    From the definition of the maps $\varphi_i$, it follows directly that every $I \in \mathcal{P}$ is increasing. Furthermore, if $I, J \in \mathcal{P}$ coincide on at least $k$ coordinates, then there must exist two codewords $\underline{c}, \underline{c}' \in \cC$ that agree on at least $k$ coordinates (since each $\varphi_i$ is a bijection), contradicting the assumption that $\cC$ is an MDS code. This proves the first part of the proposition, because $|\mathcal{P}| = |\cC| = p^k$.

    The encoding algorithm is given next. Let $\underline{m} \in \F_p^k$ be a message. 
    \begin{enumerate}
        \item Encode it as a codeword $(c_1, \ldots,c_{\ell}) \in \cC \subseteq\F_p^{\ell}$.
        \item Apply each $\varphi_i$ coordinate-wise for $i \in [\ell]$ to obtain an increasing sequence $I = (\varphi_1(c_1), \ldots, \varphi_{\ell}(c_{\ell}))$.
        \item Construct and output the generator matrix of $\RS_{\ell, k}(\underline{\alpha}_I)$. Since this is a Reed–Solomon code, this amounts to computing the corresponding Vandermonde matrix.
    \end{enumerate}
    From the preceding discussion, the encoding function defined by this algorithm is injective and yields a valid representation of a petal in our $(\ell, k, q; 1)$ RS sunflower. 
    
    We now analyze the computational complexity of the encoding algorithm. 
    In Step $1$, the algorithm encodes a message via a linear code of dimension $k$ and block length $\ell$, which takes at most $O(k\cdot \ell)$ operations in $\F_p$. 
    Step $2$ performs at most $O(\ell)$ arithmetic operations over $\mathbb{Z}$, with all involved integers bounded by $n$. 
    In Step $3$, one must compute $1, \beta, \beta^2, \ldots, \beta^{k-1}$ for each $\beta$ occurring in $\underline{\alpha}_I$, which can be carried out using $O(k\cdot \ell)$ operations in $\F_q$. Since $p \leq n/\ell < n \leq q$, we get that the overall complexity is dominated by Step $3$.
\end{proof}

\begin{remark}
    We note that the input to our encoding algorithm is a vector $\underline{m}\in \mathbb{F}_p^k$, whereas the output is a matrix in $\mathbb{F}_q^{k\times \ell}$. Thus, even writing down the output requires $\Omega(k\ell \log q)$ bits. 
    In our regime of interest, where $\ell=\Theta(k)$, the input length is $N=\Theta(k\log p)$, while the output length is $\Theta(k^2\log q)$. 
    Hence, if $q=2^{N^{\omega(1)}}$, then the output size itself is super-polynomial in the input length. 
    Consequently, regardless of the number of field operations performed by the encoding algorithm, merely writing down the output matrix already requires super-polynomial time.
\end{remark}

We next instantiate \Cref{prop:from-del-to-subspaces} using known deletion-correcting RS codes. 
For RS codes correcting substitutions, one can typically take $q=\Theta(n)$. However, when considering $k$-dimensional RS codes that can correct $n-\ell$ deletions, we do not know what is the minimal $q$ for which such a code exists.
In \cite[Theorem 16]{con2023reed}, the authors showed that for prime $q = \Theta(n^{4k-2})$, there exists an $[n,k]_q$ RS code that can correct $n-2k+1$ deletions. In fact, the proof of \cite[Theorem 16]{con2023reed} constructs an evaluation vector satisfying the rank condition of \Cref{prop:from-del-to-subspaces}. 
This leads to the following corollary

\begin{theorem} \label{thm:del-to-subs-exist-1}
    Let $k\geq 2$ be an integer. For every integer $n > 2(2k-1)^2$, there is a prime $q=\Theta(n^{4k-2})$ and an $(2k-1, k, q; 1)$ RS sunflower with $\Omega_k(n^k) = \Omega_k(q^{k/(4k-2)})$ petals.
\end{theorem}

In \cite[Theorem 24]{con2023reed}, for $k < \log n/\log\log n$, the authors constructed an explicit $[n,k]_q$ RS code where $q = n^{k^{O(k)}}$ that can correct $n-2k+1$ deletions. 
Since the field size in this construction is very large, in order to ensure that the encoding algorithm has polynomial bit complexity, we take $k$ to be constant. In this case, the input length is $O(\log n)$ and the output length is $O(\log n)$. Moreover, performing $O(k^2) = O(1)$ operations over $\Fq$ takes at most $\textrm{poly}(\log q) = \textrm{poly} (\log n)$ time.   
This implies the following corollary
\begin{theorem} \label{thm:del-to-subs-explicit}
    Let $k$ be fixed and let $n$ be such that $n > 2(2k-1)^2$ and $k < \log n/\log\log n$. There exists an explicit $(2k-1, k, q;1)$ RS sunflower where $q = n^{k^{O(k)}}$ with $\Omega (q^{1/k^{O(k)}})$ petals.
\end{theorem}

    The number of petals we obtain in these results is not optimal when compared with known sunflower codes; however, our constructions are restricted to RS sunflowers.
    Specifically, by \Cref{rem:sunflower-subspace-codes}, we know that there exist $(2k-1, k, q; 1)$ sunflower codes of size $(q^{2k - 2} - 1)/(q^{k-1}-1) = \Theta(q^{k-1})$. 
    This raises the natural question of whether our constructions can be improved. More concretely, can we build $(2k-1, k, q;1)$ RS sunflowers with larger sizes? What can be said for general $\ell > 2k-1$?

    In the next section, we give a partial affirmative answer to this question by constructing $(\ell, k, q;1)$ RS sunflowers of size $\Omega_{k,\ell}(q^{\lfloor \ell / (2k-1)\rfloor})$. Plugging $\ell = 2k-1$, we get an $(\ell,k,q;1)$ RS sunflower with $\Omega_k(q)$ petals, improving the results obtained in \Cref{thm:del-to-subs-exist-1} and \Cref{thm:del-to-subs-explicit}.

    \begin{remark}
        In \cite[Theorem 4]{CGLZ25}, the authors showed that, for $\varepsilon\in(0,1)$, random RS codes over fields of size $q\geq n+2^{\operatorname{poly}(1/\varepsilon)}k$ correct at least $(1-\varepsilon)n-2k+1$ insdel errors with high probability. This result also applies to vanishing rates. However, in this case, $\ell=2k-1+\lfloor\varepsilon n\rfloor$. Under our requirement $n>2\ell^2$, $\varepsilon$ must tend to zero as $n\to\infty$, so the theorem does not retain its linear-alphabet guarantee with a constant independent of $n$.
    \end{remark}

\subsection{Recursive construction}
In this section, we describe a deterministic recursive procedure that constructs an $(\ell,k,q;1)$ RS sunflower with $\Omega_{\ell,k}(q^{\lfloor\ell/(2k-1)\rfloor})$ petals.
We begin with a high-level overview before stating the formal algorithm.

We set $t=\lfloor\ell/(2k-1)\rfloor$ and divide the first $t(2k-1)$ coordinates into $t$ consecutive blocks, each of size $2k-1$; the remaining coordinates form a suffix.
In the first step, we greedily construct a family $\mathcal{A}_1$ of evaluation vectors whose first blocks satisfy a pairwise nonsingularity condition: for any two vectors $\underline{\alpha}, \underline{\alpha}'$ in the family, the matrix
\[
V_{k,2k-1}\big((\alpha_1,\ldots,\alpha_{2k-1}), (\alpha_1',\ldots,\alpha_{2k-1}')\big),
\]
formed from the $(2k-1)$-coordinate prefixes of these vectors, is nonsingular.

In each subsequent step $s$, every vector $\underline{z}\in \mathcal{A}_{s-1}$ generates its own family of descendants. Every descendant of $\underline{z}$ keeps the first $(s-1)\cdot (2k-1)$ coordinates identical to those of $\underline{z}$, while its $s$-th block is chosen greedily so that the associated $V_{k,2k-1}$ matrices remain nonsingular. The coordinates beyond the $s$-th block are then assigned arbitrarily, with the only constraint that the resulting vector has pairwise distinct entries.

We next present the formal algorithm.

\begin{framed}
    \textbf{\underline{Algorithm 1}}\\
    
    \noindent\textbf{Input:} Let $k,\ell$ be integers such that $k\geq 2$ and $\ell\geq 2k-1$, and set $t=\lfloor\ell/(2k-1)\rfloor$. Let $\underline{z}= (z_1, \ldots,z_{\ell}) \in \Fq^{\ell}$ be a vector with pairwise distinct entries.

    \noindent\textbf{Output:} A set $\mathcal{A} \subset \Fq^{\ell}$ of evaluation vectors. such that 
    \[
    \RS_{\ell, k}(\underline{\alpha}) \cap \RS_{\ell, k}(\underline{\beta}) = \langle \underline{1}\rangle _{\Fq}
    \]
    for all distinct $\underline{\alpha}, \underline{\beta}\in \mathcal{A}$.

    \paragraph{Step 1:} Set $\mathcal{A}_1 = \{ \underline{z} \}$. 
    
    Repeat: 
    Let $\mathcal{B} \subseteq \Fq^{2k-1}$ be the set of all vectors $\underline{x}\in \Fq^{2k-1}$ for which one of the following holds:
    \begin{itemize}
        \item There exists $\underline{y}\in \mathcal{A}_1$ for which
            \[
            \det(V_{k, 2k-1}((y_1, \ldots,y_{2k-1}), (x_1, \ldots, x_{2k-1}))) = 0 
            \]
        \item There exists distinct $i,j \in [2k-1]$ for which $x_i = x_j$.
    \end{itemize}
    If $\mathcal{B} \neq \Fq^{2k-1}$, pick $\underline{u}\in \Fq^{2k-1} \setminus \mathcal{B}$ and pick $\underline{u}'\in \Fq^{\ell - (2k-1)}$ such that $\underline{u} \circ \underline{u}'$ is a vector with pairwise distinct coordinates. Append $\underline{u} \circ \underline{u}'$ to $\mathcal{A}_1$. 
    Otherwise, i.e., if $\mathcal{B} = \Fq^{2k-1}$, exit loop.
    
    \paragraph{Step 2:} For each $s = 2,\ldots, t$ do:

    For each $\underline{z} \in \mathcal{A}_{s-1}$, we will construct $\mathcal{A}_{s}^{(\underline{z})}$ as follows: 
    
    Set $\mathcal{A}_{s}^{(\underline{z})} = \{\underline{z}\}$. 
    
    Repeat: Let $\mathcal{B}\subseteq \Fq^{2k-1}$ be the set of all vectors $\underline{x}\in \Fq^{2k-1}$ for which one of the following holds:
    \begin{itemize}
        \item There exists $\underline{y}\in \mathcal{A}_{s}^{(\underline{z})}$ for which 
        \[
        \det(V_{k,2k-1}((y_{(s-1)(2k-1)+1}, \ldots, y_{s(2k-1) }), (x_1, \ldots, x_{2k-1}))) = 0
        \]
        \item There exists distinct $i,j \in [2k-1]$ for which $x_i = x_j$
        \item We have that $\{x_1, \ldots,x_{2k-1}\} \cap \{z_1, \ldots,z_{(s-1)(2k-1)} \} \neq \emptyset$.
    \end{itemize}
    If $\mathcal{B} \neq \Fq^{2k-1}$, pick $\underline{u}\in \Fq^{2k-1} \setminus \mathcal{B}$ and $\underline{u}'\in \Fq^{\ell - s(2k-1)}$ such that the vector $(z_1, \ldots,z_{(s-1)(2k-1)}) \circ \underline{u} \circ \underline{u}'$ has pairwise distinct coordinates and then
    add it to  $\mathcal{A}_{s}^{(\underline{z})}$.  
    Otherwise, i.e., if $\mathcal{B} = \Fq^{2k-1}$, exit the current loop and continue to the next $\underline{z}\in \mathcal{A}_{s-1}$.

    When we iterated over all $\underline{z}\in \mathcal{A}_{s-1}$, set
    \[
    \mathcal{A}_s = \bigcup_{\underline{z}\in \mathcal{A}_{s-1}} \mathcal{A}_{s}^{(\underline{z})}
    \]
    and continue to the next $s$.

    \paragraph{Step 3:}
    Output $\mathcal{A} =\mathcal{A}_t$.
\end{framed}

Our goal now is to prove the correctness of this algorithm and to lower bound the size of its output.
We first prove the following simple lemma

\begin{lemma}\label{lem:pol}
Let  $\alpha_1,\ldots,\alpha_{2k-1}$ be pairwise distinct elements of $\F_q$ and let $X_1, \ldots,X_{2k-1}$ be formal variables. Define the polynomial
\[ P(X_1,\ldots,X_{2k-1}) :=
    \det\left(\begin{array}{cccccccccc}
    1 & \alpha_1 & \alpha_1^2 & \cdots & \alpha_1^{k-1} & X_1 & X_1^2 & \cdots & X_1^{k-1} \\
    1 & \alpha_{2} & \alpha_{2}^2 & \cdots & \alpha_{2}^{k-1} & X_2 & X_2^2 & \cdots & X_2^{k-1} \\
     \vdots & & & & & & & & \vdots \\
    1 & \alpha_{{2k-1}} & \alpha_{{2k-1}}^2 & \cdots & \alpha_{{2k-1}}^{k-1} & X_{{2k-1}} & X_{{2k-1}}^2 & \cdots & X_{{2k-1}}^{k-1} \\
    \end{array}\right),
    \]
    over the ring $\Fq[X_1, \ldots, X_{2k-1}]$. Then, $P(X_1, \ldots, X_{2k-1})$ is a nonzero polynomial of total degree $k(k-1)/2$. 
    In particular, the number of solutions to $P(X_1,\ldots,X_{2k-1})=0$ in $\F_q^{2k-1}$ is at most 
\[ \frac{k(k-1)}2 q^{2k-2}. \]
\end{lemma}
\begin{proof}
    Consider the monomial $X_1^{k-1}X_2^{k-2}\cdots X_{k-1}$. It is easy to see that its coefficient in $P(X_1, \ldots, X_{2k-1})$, up to sign, is 
    \[ 
    \det\left(\begin{array}{cccccccccc}
    1 & \alpha_k & \alpha_k^2 & \cdots & \alpha_k^{k-1} \\
    1 & \alpha_{k+1} & \alpha_{k+1}^2 & \cdots & \alpha_{k+1}^{k-1} \\
     \vdots & & & & \vdots \\
    1 & \alpha_{{2k-1}} & \alpha_{{2k-1}}^2 & \cdots & \alpha_{{2k-1}}^{k-1} 
    \end{array}\right) 
    \]
    which is nonzero since it is the determinant of a Vandermonde matrix on pairwise distinct elements of $\F_q$. 
    Therefore, $P$ is a nonzero polynomial and $\deg(P)\geq {k(k-1)}/2$.
    Moreover, every term in the determinant uses each of the $k-1$ variable columns exactly once. 
    Hence every such term has total degree $\sum_{i=1}^{k-1}i = \binom{k}{2}$.
    The bound on the number of zeros now follows from the Schwartz–Zippel lemma \cite[Corollary 1]{schwartz1980fast}.
\end{proof}

We are now ready to prove the correctness of our algorithm and also provide an estimate of the number of vectors it constructs.

\begin{theorem}\label{thm:recursive-construction}
    Let $k$ and $\ell$ be positive integers such that $k \geq 2$ and $\ell\geq 2k-1$, and set $t=\lfloor\ell/(2k-1)\rfloor$.
    Assume that $q > \binom{2k-1}{2} + (2k-1)\cdot \ell$ is a prime power.
    Then, applying Algorithm 1 on any input vector $\underline{z}$ with pairwise distinct coordinates, outputs a set $\mathcal{A}$ of evaluation vectors such that $\RS_{\ell,k}(\underline{\alpha}) \cap \RS_{\ell,k}(\underline{\alpha}') = \langle \underline{1}\rangle_{\Fq}$ for all distinct $\underline{\alpha}, \underline{\alpha}'\in \mathcal{A}$. Furthermore,
    \[
    |\mathcal{A}| \geq \left( \frac{q - \binom{2k-1}{2} - (2k-1)^2(t-1)}{\binom{k}{2}} \right)^t \;.
    \]
    In other words, for fixed $k$ and $\ell$, Algorithm 1 constructs an $(\ell,k,q;1)$ RS sunflower with $\Omega_{\ell,k} (q^{\lfloor\ell/(2k-1)\rfloor})$ petals.
\end{theorem}
\begin{proof}
    For every $s\in [t]$ and every two distinct $\underline{\alpha}, \underline{\alpha}' \in \mathcal{A}_s$ the following holds:
    \begin{enumerate}
        \item They have distinct prefixes of length $s\cdot (2k-1)$. Namely,
        \[
        \left(\alpha_1, \ldots, \alpha_{s(2k-1)} \right) \neq \left( \alpha_1', \ldots, \alpha_{s(2k-1)}'\right) \;.
        \]
        \item There exists $r\in [s]$ such that the first $(r-1)(2k-1)$ coordinates of $\underline{\alpha}$ and $\underline{\alpha}'$ coincide and 
        \[
        \det\left(V_{k,2k-1} \left( \left(\alpha_{(r-1)(2k-1) + 1}, \ldots, \alpha_{r(2k-1)}\right), \left(\alpha_{(r-1)(2k-1)+1}', \ldots, \alpha_{r(2k-1)}'\right) \right) \right) \neq 0 \;.
        \]
    \end{enumerate}
    Indeed, if $\underline{\alpha}, \underline{\alpha}'$ belong to the same family $\mathcal{A}_{s}^{\underline{z}}$ for some $\underline{z}$, then, it must be that 
    \[
    \det\left(V_{k,2k-1} \left( \left(\alpha_{(s-1)(2k-1) + 1}, \ldots, \alpha_{s(2k-1)}\right), \left(\alpha_{(s-1)(2k-1)+1}', \ldots, \alpha_{s(2k-1)}'\right) \right) \right) \neq 0 \;,
    \]
    which implies that 
    \[
    \left(\alpha_{(s-1)(2k-1)+1}, \ldots, \alpha_{s(2k-1)} \right) \neq \left( \alpha_{(s-1)(2k-1)+1}', \ldots, \alpha_{s(2k-1)}'\right)
    \]
    by the definition of the matrix $V_{k,2k-1}$ given in \Cref{matrix}. This implies that the prefixes of $\underline{\alpha}$ and $\underline{\alpha}'$ of length $s(2k-1)$ are distinct.
    If $\underline{\alpha}$ belong to families rooted at distinct $\underline{z},\underline{z}' \in \mathcal{A}_{s-1}$, they were already separated at an earlier block of length $2k-1$ by induction.
    The block in item~2 is a nonsingular $(2k-1)\times(2k-1)$ row submatrix of $V_{k,\ell}(\underline{\alpha},\underline{\alpha}')$. Hence, this matrix has rank $2k-1$, and \Cref{thm:constr} shows that $\mathcal{A}$ yields an $(\ell,k,q;1)$ RS sunflower of size $|\mathcal{A}|$.
    
    Our next objective is to prove the lower bound on the size of $\mathcal{A}$. 
    We first lower bound the number of evaluation vectors generated in Step~$1$. 
    Namely, we lower bound the size of $\mathcal{A}_1$.
    Clearly, the number of vectors $\underline{x}\in \Fq^{2k-1}$ such that there are distinct $i,j\in[2k-1]$ for which $x_i = x_j$ is at most $\binom{2k-1}{2} \cdot q^{2k-2}$. 
    Additionally, by \Cref{lem:pol}, for a fixed $(y_1, \ldots,y_{2k-1})\in \Fq^{2k-1}$, the number of solutions to \[
    \det(V_{k,2k-1} (y_1, \ldots,y_{2k-1}), (X_1, \ldots,X_{2k-1})) = 0
    \]
    is at most $\binom{k}{2} q^{2k-2}$. Therefore, at the worst case scenario, every vector we add to $\mathcal{A}_1$ increases the set of bad vectors $\mathcal{B}$ by at most $\binom{k}{2} q^{2k-2}$.
    Furthermore, note that as long as $\mathcal{B} \neq \Fq^{2k-1}$ and we pick any $u \in \Fq^{2k-1} \setminus \mathcal{B}$, we can always choose some $\underline{u}' \in \Fq^{\ell - (2k-1)}$ so that the concatenation $\underline{u} \circ \underline{u}'$ has all coordinates distinct. The reason is that $q > \ell$, which guarantees a sufficient number of field elements to avoid any repeated coordinates.
    The process ends when $|\mathcal{B}| = q^{2k-1}$, and thus, by the union bound, 
    \[
    \binom{2k-1}{2} q^{2k-2} + |\mathcal{A}_1| \binom{k}{2} q^{2k-2} \geq |\mathcal{B}| = q^{2k-1} \;.
    \]
    This implies that 
    \[
    |\mathcal{A}_1| \geq \frac{q - \binom{2k-1}{2}}{\binom{k}{2}}\;.
    \]
    
    Our next goal is to lower bound $\mathcal{A}_s$ for every $s\in [2,t]$. Again, note that the number of vectors $\underline{x}\in \Fq^{2k-1}$ such that there are distinct $i,j\in[2k-1]$ for which $x_i = x_j$ is at most $\binom{2k-1}{2} \cdot q^{2k-2}$. 
    However, by the third condition, the set $\mathcal{B}$ also contains all $\underline{x}\in \Fq^{2k-1}$ where $\{x_1, \ldots, x_{2k-1}\} \cap \{z_1, \ldots, z_{(s-1)(2k-1)}\} \neq \emptyset$ where $\underline{z}$ is fixed. 
    This adds at most $(2k-1)\cdot (s-1)(2k-1)\cdot q^{2k-2}$ vectors to $\mathcal{B}$. Therefore, any vector that is added to $\mathcal{A}_s$ increases the number of bad vectors by at most $\binom{k}{2}q^{2k-2}$.
    Finally, as before, for any fixed $(y_1, \ldots,y_{2k-1})\in \Fq^{2k-1}$, the number of solutions to $\det(V_{k,2k-1}((y_1, \ldots,y_{2k-1}), (X_1, \ldots, X_{2k-1})))=0$ is at most $\binom{k}{2} q^{2k-2}$. 
    Moreover, as in the previous case, because $q > \ell$, we can always choose some $\underline{u}' \in \Fq^{\ell - s(2k-1)}$ so that the vector we append to $\mathcal{A}_s^{\underline{z}}$ has no repeated entries in its coordinates.
    We stop adding elements to $\mathcal{A}_{s}^{(\underline{z})}$ when $|\mathcal{B}| = q^{2k-1}$. Thus,
    \[
    \binom{2k-1}{2} q^{2k-2} + (2k - 1)^2 (s-1)q^{2k-2} + \left|\mathcal{A}_{s}^{(\underline{z})}\right| \binom{k}{2}q^{2k-2} \geq q^{2k-1} \;.
    \]
    Rearranging, we get that 
    \[
    \left|\mathcal{A}_{s}^{(\underline{z})}\right| \geq \frac{q - \binom{2k-1}{2} - (2k-1)^2(s-1)}{\binom{k}{2}}\;.
    \]
    Now, since for any two distinct $\underline{z}, \underline{z}'\in \mathcal{A}_{s-1}$, by our construction, we have that $\mathcal{A}_{s}^{(\underline{z})} \cap \mathcal{A}_{s}^{(\underline{z}')} = \emptyset$. Indeed, recall that every vector in $\mathcal{A}_{s}^{(\underline{z})}$ has $(z_1, \ldots,z_{(s-1)(2k-1)})$ as its prefix which is distinct from any other prefix of other $\underline{z}'\in \mathcal{A}_{s-1}$.
    Thus, 
    \[
    |\mathcal{A}_s| \geq |\mathcal{A}_{s-1}| \cdot \frac{q - \binom{2k-1}{2} - (2k-1)^2(s-1)}{\binom{k}{2}} \;.
    \]

    Thus, we have
    \[
    \left|\mathcal{A}_t\right| \geq \prod_{s=1}^{t} \frac{q - \binom{2k-1}{2} - (2k-1)^2(s-1)}{\binom{k}{2}} \geq \left( \frac{q - \binom{2k-1}{2} - (2k-1)^2(t-1)}{\binom{k}{2}} \right)^t\;,
    \]
    and the theorem follows by the choice of $t$.
\end{proof}

We now argue about the complexity of our recursive construction.

\begin{proposition}
    Algorithm $1$ can be implemented deterministically using 
    \[
    O_{\ell, k}\left(|\mathcal{A}| \cdot q^{2k-1} \right)
    \]
    field operations over $\Fq$ and equality tests.
\end{proposition}
\begin{proof}
    For two vectors $\underline{u}, \underline{v} \in \Fq^{2k-1}$, verifying whether $\det(V_{k,2k-1}(\underline{u}, \underline{v})) = 0$ can be done in time $O(k^3)$.
 Specifically, we first precompute all powers $u_i^j, v_i^j$ for every $i \in [2k-1]$ and $j \in [k-1]$ in $O(k^2)$ time, and then evaluate the determinant of the resulting $(2k-1)\times(2k-1)$ matrix, which requires $O(k^3)$ time.

Given a vector $\underline{x} \in \Fq^{2k-1}$, checking whether $x_i = x_j$ for some distinct $i \neq j$ can be done with $O(k^2)$ equality comparisons. Moreover, testing whether $\underline{x}$ contains a coordinate that already occurs in a prefix of $\underline{z}$ can be performed in $O(\ell \cdot k)$ equality tests.

A naive exhaustive search over all $q^{2k-1}$ vectors in $\Fq^{2k-1}$, combined with checking all three conditions, would therefore require $O((k^3 + \ell \cdot k) q^{2k-1})$ field operations and equality tests.

Observe, however, that when constructing $\mathcal{A}_{s}^{(\underline{z})}$ (or $\mathcal{A}_1$ in the initial step), we do not recompute $\mathcal{B}$ from the beginning each time a new vector is added. 
Instead, we keep a record of all vectors that have previously failed at least one of the conditions; for these vectors we skip the determinant computation in subsequent steps.

Consequently, the overall running time to construct $\mathcal{A}$ is bounded by
\[
O\left(\sum_{j=1}^{t} |\mathcal{A}_j| \cdot (k^3 + \ell \cdot k) q^{2k-1}\right) \leq O\left(t \, |\mathcal{A}_t| \cdot (k^3 + \ell k) q^{2k-1}\right)\;,
\]
where the inequality follows from $|\mathcal{A}_j|\leq|\mathcal{A}_t|$. Since $t=\lfloor\ell/(2k-1)\rfloor$, the proposition follows.
\end{proof}

\subsection{Existence of subspace MDS-codes and RS sunflowers}

In this subsection, we prove two existence results. First, we construct a large family of MDS codes whose pairwise intersections have dimension at most $u$. Then, we apply the same argument within the family of RS codes to obtain an $(n,k,q;1)$ RS sunflower with $\Omega_{k,n}(q^{n-2k+2})$ petals. 
Throughout this subsection, $k\geq2$ and $n\geq2k-1$ are fixed integers, $q\geq n$ is a prime power, and all asymptotics are taken as $q\to\infty$.

Let $\cG_q(k,n)$ be the Grassmannian, i.e., the set of subspaces of dimension $k$ in $\F_q^n$. In the sequel, we let $$\qbin{a}{b}{q}= \prod_{i=0}^{b-1}\frac{\left(q^a-q^i\right)}{\left(q^b-q^i\right)}$$
be the $q$-binomial coefficient of integers $a \ge b \ge 0$; see e.g.~\cite{stanley2011enumerative}. It is well-known that $\qbin{a}{b}{q}$ counts the number of distinct $b$-dimensional subspaces of an $a$-dimensional vector space over~$\F_q$. Therefore, $|\cG_q(k,n)|=\qbin{n}{k}{q}$.
We use the convention $\qbin{a}{b}{q}=0$ when $b<0$ or $b>a$.

We denote by $\cM_q(k,n)$ the set of MDS codes in $\cG_q(k,n)$ and write $M_q(k, n) := |\cM_q(k,n)|$.

We first borrow the following theorem from \cite{gruica2022common}.

\begin{theorem}[{Theorem 3.8 in \cite{gruica2022common}}]\label{thm:cone-count}
    Let $K\subseteq\Fq^n$ be closed under multiplication by scalars in $\Fq$. Furthermore, assume that $|K| \geq q$, $n\geq 3$ and $1 \leq k\leq n-1$.
    Denote by $\mathcal{F}_k^K$ the collection of all subspaces of $\Fq^n$ of dimension $k$ that intersect $K$, i.e., $K \cap W \neq \{\underline{0}\}$. Then,
    \[
    \frac{\frac{|K|-1}{q-1} \cdot \qbin{n-1}{k-1}{q}}{1 + \left( \frac{|K|-1}{q-1} - 1\right) \left( \frac{q^{k-1} - 1}{q^{n-1} - 1}\right)} \leq \left| \mathcal{F}_k^K\right| \leq \frac{|K|-1}{q-1} \cdot \qbin{n-1}{k-1}{q}
    \]
\end{theorem}

By setting $K$ appropriately in the theorem, we get upper and lower bounds on the number of MDS codes of dimension $k$.
\begin{lemma} \label{lem:mds}
Let $b_q(n,r):=\sum_{i=0}^{r}\binom{n}{i}(q-1)^i$ denote the size of a Hamming ball of radius $r$ in $\F_q^n$. We have
\begin{align*}
   \qbin{n}{k}{q}-\frac{b_q(n,n-k)-1}{q-1}\qbin{n-1}{k-1}{q}\leq M_q(k,n)&\leq\qbin{n}{k}{q}-\frac{\frac{b_q(n,n-k)-1}{q-1}\qbin{n-1}{k-1}{q}}{1+\left(\frac{b_q(n,n-k)-1}{q-1}-1 \right)\left( \frac{q^{k-1}-1}{q^{n-1}-1}\right)}.
\end{align*}

\end{lemma}
\begin{proof}
    Let $K$ be the set of vectors in $\Fq^n$ with at least $k$ zero coordinates. Then $|K|=b_q(n,n-k)\geq q$, since $n-k\geq1$, and $K$ is closed under scalar multiplication. An $[n,k]_q$ code is MDS if and only if every nonzero codeword has at most $k-1$ zero coordinates. Consequently, $ M_q(k,n)=\qbin{n}{k}{q}-|\mathcal{F}_k^K|$.
    The bounds now follow from Theorem~\ref{thm:cone-count}.
\end{proof}

For fixed $n$ and $k$, the term subtracted in the lower bound of Lemma~\ref{lem:mds} is $O_{k,n}(q^{k(n-k)-1})$. Thus, the standard asymptotics of Gaussian coefficients, namely $\qbin{a}{b}{q}=q^{b(a-b)}(1+o(1))$, together with the trivial upper bound $M_q(k,n)\leq\qbin{n}{k}{q}$, give
\[
M_q(k,n)=q^{k(n-k)}\left(1+o(1)\right).
\]
We next recall the number of subspaces having a prescribed intersection dimension with a fixed subspace.

\begin{lemma}[{Theorem 5 in \cite{koetter2008coding}}] \label{lem:injball}
For $\cC\in \cG_q(k,n)$, the number of $\cD \in \cG_q(k,n)$ with $\dim(\cC \cap \cD)=r$ is
\begin{align*}
    q^{(k-r)^2}\qbin{k}{r}{q}\qbin{n-k}{k-r}{q}.
\end{align*}
\end{lemma}

For convenience, define
\[
A_q(k,n,u) := \sum_{i=0}^{k-u-1}
q^{i^2}\qbin{k}{i}{q}\qbin{n-k}{i}{q} -1.
\]
By Lemma~\ref{lem:injball}, $A_q(k,n,u)$ is the number of subspaces $\cD\ne\cC$ with $\dim(\cC\cap\cD)\geq u+1$, for a fixed $\cC\in\cG_q(k,n)$. The subtraction of $1$ excludes $\cC$ itself.
Note that for $u=k-1$, one has
$A_q(k,n,k-1)=0$, since two distinct $k$-dimensional subspaces cannot
intersect in dimension $k$. Thus, we will assume that $0\leq u\leq k-2$.

\begin{claim} \label{clm:A-q-k-n-u}
    Let $u,k,n$ be fixed integers such that $0\leq u\leq k-2$ and $n\geq2k-1$. As $q\to\infty$, we have
    \[
    A_q(k,n,u) = q^{(k-u-1)(n-k+u+1)}(1+o(1))\;.
    \]
\end{claim}
\begin{proof}
    It is well known that $\qbin{a}{b}{q} = q^{b(a - b)}(1 + o(1))$ where the $o(1)$ tends to $0$ as $q\to +\infty$. Thus, the summand indexed by $i$ in
    $A_q(k,n,u)$ has asymptotic order
    \[
    q^{i^2+i(k-i)+i(n-k-i)}(1 + o(1)) = q^{i(n-i)} (1 + o(1)).
    \]
    Since $u\leq k-2$, at least one summand remains. Moreover, $i(n-i)$ is strictly increasing for integers $0\leq i\leq k-u-1\leq k-1$ when $n\geq2k-1$. Hence, the dominant term corresponds to $i=k-u-1$.
    Therefore,
    \begin{equation*}
        A_q(k,n,u) = q^{(k-u-1)(n-k+u+1)}(1+o(1)).
    \end{equation*}
\end{proof}

We now apply greedy selection to obtain a family of MDS codes with small pairwise intersections.
\begin{proposition}\label{prop:MDS-lower-bound}
Let $k\geq2$, $n\geq2k-1$, and $0\leq u\leq k-2$ be fixed integers, and let $q\geq n$ be a prime power. There exists a set $\mathcal{M}\subseteq\mathcal{M}_q(k,n)$ such that $\dim(\cC\cap\cD)\leq u$ for every two distinct $\cC,\cD\in\mathcal{M}$ and
\[
|\mathcal{M}|\geq q^{(u+1)(n-2k+u+1)}(1+o(1)).
\]
\end{proposition}
\begin{proof}
Set $\mathcal{M}=\emptyset$ and let $\mathcal{T}=\mathcal{M}_q(k,n)$ be the set of remaining codes. At each step, select $\cC\in\mathcal{T}$, add it to $\mathcal{M}$, and remove from $\mathcal{T}$ every code $\cD$ with $\dim(\cC\cap\cD)\geq u+1$. Stop when $\mathcal{T}$ is empty.

Each step removes at most $A_q(k,n,u)+1$ codes, including the selected code. Therefore,
\[
M_q(k,n)\leq|\mathcal{M}|\bigl(A_q(k,n,u)+1\bigr),
\]
which gives the stated finite bound. By construction, any two distinct codes in $\mathcal{M}$ intersect in dimension at most $u$. Finally, Lemma~\ref{lem:mds} and Claim~\ref{clm:A-q-k-n-u} give
\begin{align*}
\frac{M_q(k,n)}{A_q(k,n,u)+1}
&=\frac{q^{k(n-k)}(1+o(1))}{q^{(k-u-1)(n-k+u+1)}(1+o(1))}\\
&=q^{(u+1)(n-2k+u+1)}(1+o(1)).
\end{align*}
\end{proof}

\begin{remark}
The quantity $A_q(k,n,u)$ counts all $k$-dimensional subspaces whose intersection with a fixed subspace has dimension at least $u+1$. 
A smaller upper bound on the number of such subspaces that are MDS codes would improve the lower bound in Proposition~\ref{prop:MDS-lower-bound}.
\end{remark}

We next restrict the greedy argument to RS codes. Let $\operatorname{RS}_q(k,n)$ denote the set of distinct $[n,k]_q$ RS codes and write $R_q(k,n)=|\operatorname{RS}_q(k,n)|$. By Corollary~\ref{RS-count},
\[
R_q(k,n)=\frac{(q-2)!}{(q-n)!}=q^{n-2}(1+o(1)).
\]
The following claim bounds the number of RS codes whose intersection with a fixed RS code has dimension at least two. For fixed $k\geq3$, this bound has smaller asymptotic order than $A_q(k,n,1)$.
\begin{claim}\label{clm:intersecting-rs}
Let $k\geq2$ and $2k-1\leq n\leq q$. Fix an $[n,k]_q$ RS code $C$. The number of $[n,k]_q$ RS codes $D$ with $\dim(C\cap D)\geq2$, including $D=C$, is at most
\[
B_q(k,n):=2(k-1)^{n-2}q^{2k-4}.
\]
\end{claim}
\begin{proof}
Fix an evaluation vector $\underline{\alpha}=(\alpha_1,\ldots,\alpha_n)$ for $C$. By Lemma~\ref{RS-eq}, every code $D$ has a unique evaluation vector of the form $\underline{\beta}=(0,1,\beta_3,\ldots,\beta_n)$, obtained by an affine change of its evaluation points.

If $\dim(C\cap D)\geq2$, the intersection contains a nonconstant codeword, so there are nonconstant polynomials $f,g\in\Fq[X]$ of degree at most $k-1$ such that $
f(\alpha_i)=g(\beta_i)$ for every $i\in[n]$.
Subtracting $f(\alpha_1)$ from both polynomials and dividing both by the leading coefficient of $f$, we may assume that $f$ is monic and $f(\alpha_1)=0$. The number of such nonconstant polynomials $f$ is
\[
\sum_{j=0}^{k-2}q^j=\frac{q^{k-1}-1}{q-1}\leq2q^{k-2}.
\]
For each $f$, the polynomial $g$ satisfies $g(0)=0$ and $g(1)=f(\alpha_2)$. These are two independent linear constraints on the $k$ coefficients of $g$, so there are $q^{k-2}$ solutions of degree at most $k-1$, and at most this many nonconstant solutions.

For each such pair of nonconstant polynomials $(f,g)$ and each $i\geq3$, the polynomial $g(X)-f(\alpha_i)$ is nonzero and has degree at most $k-1$. Hence, there are at most $k-1$ choices for $\beta_i$, and at most $(k-1)^{n-2}$ choices for $(\beta_3,\ldots,\beta_n)$. Discarding tuples with repeated coordinates can only decrease this count. Multiplying the bounds proves the claim.
\end{proof}

We can now construct an RS sunflower by greedy selection.
\begin{proposition}\label{prop:RS-lower-bound}
Let $k\geq2$ and $n\geq2k-1$ be fixed integers, and let $q\geq n$ be a prime power. There exists an $(n,k,q;1)$ RS sunflower $\mathcal{S}$ with center $\langle\underline{1}\rangle_{\Fq}$ and
\[
|\mathcal{S}|\geq\frac{1}{2(k-1)^{n-2}}q^{n-2k+2}(1+o(1))
=\Omega_{k,n}(q^{n-2k+2}).
\]
\end{proposition}
\begin{proof}
Start with all codes in $\operatorname{RS}_q(k,n)$. Repeatedly select a remaining code and discard every remaining code whose intersection with it has dimension at least two. By Claim~\ref{clm:intersecting-rs}, each step discards at most $B_q(k,n)$ codes, including the selected code. Thus, at least $\lceil R_q(k,n)/B_q(k,n)\rceil$ codes are selected. Their pairwise intersections have dimension at most one; since every RS code contains $\langle\underline{1}\rangle_{\Fq}$, each such intersection equals this subspace. Finally,
\[
\frac{R_q(k,n)}{B_q(k,n)}
=\frac{q^{n-2}(1+o(1))}{2(k-1)^{n-2}q^{2k-4}}
=\frac{1}{2(k-1)^{n-2}}q^{n-2k+2}(1+o(1)).
\]
\end{proof}

\subsection{Comparison}

In this subsection, we compare the greedy existence bounds with the recursive construction. Throughout, $k$ and $\ell\geq2k-1$ are fixed and $q\to\infty$.

\paragraph{RS sunflowers.}
The case $k=2$ can be settled exactly. All distinct two-dimensional RS codes form an $(\ell,2,q;1)$ RS sunflower. By Corollary~\ref{RS-count}, its size is
\[
R_q(2,\ell)=\frac{(q-2)!}{(q-\ell)!}=q^{\ell-2}(1+o(1)).
\]
This is the maximum possible size, since it includes every RS code of the given length and dimension. The recursive construction gives the smaller lower bound $\Omega_{\ell}(q^{\lfloor\ell/3\rfloor})$.

Assume now that $k\geq3$. Theorem~\ref{thm:recursive-construction} gives an RS sunflower with $\Omega_{k,\ell}(q^{\lfloor\ell/(2k-1)\rfloor})$ petals. By Proposition~\ref{prop:RS-lower-bound}, with $n=\ell$, the greedy argument gives an RS sunflower with $\Omega_{k,\ell}(q^{\ell-2k+2})$ petals. The two exponents agree when $\ell=2k-1$, where both are equal to $1$. For every $\ell\geq2k$, we have $
\left\lfloor\frac{\ell}{2k-1}\right\rfloor<\ell-2k+2$,
so the greedy bound has the larger exponent.

\paragraph{MDS code families.}
For $k\geq3$, Proposition~\ref{prop:MDS-lower-bound}, with $u=1$ and $n=\ell$, gives a family $\mathcal{M}$ of $[\ell,k]_q$ MDS codes with
\[
|\mathcal{M}|\geq q^{2(\ell-2k+2)}(1+o(1))
\]
and pairwise intersection dimension at most one. Thus, $\mathcal{M}$ is a constant-dimension subspace code with minimum subspace distance at least $2k-2$. Its lower-bound exponent is larger than both exponents obtained in the RS constructions for every $\ell\geq2k-1$. 
These families need not be sunflowers: their pairwise intersections may differ, and they need not have a common one-dimensional center.

\section{Application to random linear network coding}
\subsection{Error Model for Random Linear Network Coding}

In the last section of this paper, we explain how RS sunflowers may be used in the framework of non-coherent random linear network coding. In this model, neither the transmitter nor the receiver is assumed to know the network topology or the linear combinations applied by the intermediate nodes. Packets are vectors over $\mathbb{F}_q$, and intermediate nodes forward random $\mathbb{F}_q$-linear combinations of the packets they receive. We assume throughout this section that there is a single source of information.

Suppose that the source transmits $r$ packets $x_1,\dots,x_r \in \mathbb{F}_q^n$. Collecting these packets as rows gives a matrix
\[
X \in \mathbb{F}_q^{r \times n}.
\]
The receiver collects $N$ packets, forming a matrix
\[
Y \in \mathbb{F}_q^{N \times n}.
\]
After discarding linearly dependent packets, we may assume that $\operatorname{rk}(Y)=N$.
Due to the linear operations performed within the network, the end-to-end input--output relation is modeled by
\[
Y = A X + E,
\]
where $A \in \mathbb{F}_q^{N \times r}$ is an unknown transfer matrix induced by the network, and $E \in \mathbb{F}_q^{N \times n}$ is an error matrix accounting for corrupted packets.

Packet errors may occur on arbitrary network links and propagate through the linear mixing operations. As a result, even a small number of injected errors may affect many received packets. However, since the error matrix $E$ is generated by a limited number of error injections, its rank is typically small. This observation motivates modeling network errors via rank-based and subspace-based metrics.

In the error-free case ($E=0$), and assuming that $A$ has full rank on the transmitted space, we have
\[
\operatorname{rowsp}(Y) = \operatorname{rowsp}(A X) = \operatorname{rowsp}(X),
\]
so the information transmitted through the network is naturally identified with the subspace spanned by the transmitted packets. Consequently, in the non-coherent setting, information is associated not with individual packets, but with the transmitted subspace.

More generally, if
\[
V=\operatorname{rowsp}(X)
\qquad \text{and} \qquad
U=\operatorname{rowsp}(Y),
\]
then $U$ may differ from $V$ in two ways:
\begin{itemize}
  \item Dimensions of $V$ may be lost because of rank deficiency of $A$ (\emph{erasures}), and
  \item Additional dimensions may be introduced by the error matrix $E$ (\emph{errors}).
\end{itemize}

This behavior can be abstracted by the \textbf{operator channel} model, in which both the channel input and output are subspaces of a fixed ambient space $W$ isomorphic to $\mathbb{F}_q^n$.
The received subspace can be written as
\[
U = H_z(V) \oplus E',
\]
where $H_z(V)$ is a $z$-dimensional subspace of $V$, with $z\leq \dim(V)$, modeling erasures, and $E'$ is an error subspace disjoint from $V$ modeling injected errors. The subspace distance treats insertions and deletions of dimensions symmetrically and generalizes the role of the Hamming distance in classical coding theory. This metric underlies the design and analysis of constant-dimension subspace codes and lifted rank-metric codes for error correction in random network coding.
For more details we refer to \cite{koetter2008coding,wachter2013decoding}.

We now specialize this framework to the code families constructed in the previous sections. Let
\[
\mathcal{S}=\{\mathsf{RS}_{\ell,k}(\aaa_i): i\in [s]\}
\]
be an $(\ell,k,q;1)$ RS sunflower with $2k-1\leq \ell\leq q-1$ and center $\langle \underline{1}\rangle_{\Fq}$. Thus, for all distinct $i,j\in [s]$,
\[
\mathsf{RS}_{\ell,k}(\aaa_i)\cap \mathsf{RS}_{\ell,k}(\aaa_j)=\langle \underline{1}\rangle_{\Fq}.
\]
We identify the messages with the codewords of $\mathcal{S}$. We first discuss recovery when no errors occur, and then describe a list-recovery procedure for a simple error model.

\subsection{Error-free case}

The source decides to send $\mathrm{RS}_{\ell,k}(\aaa_i) \in \mathcal{S}$ for some $i \in [s]$. It sends through the network a basis of $\mathrm{RS}_{\ell,k}(\aaa_i)$. In the case where there are no errors and no rank losses, all receivers will receive a basis of $\mathrm{RS}_{\ell,k}(\aaa_i)$ (possibly different from the one sent by the source). In order to recover the $\ell$-tuple $\aaa_i$, one can use the Sidelnikov–Shestakov attack \cite{sidelnikov1992insecurity}. 
For our purposes, the relevant feature of this algorithm is that it reconstructs, in polynomial time, the hidden generalized Reed--Solomon structure from a generator matrix by exploiting algebraic properties of generalized RS codes, especially their minimum-weight codewords.
More precisely, it reconstructs the evaluation points, up to the natural equivalences that define the same RS code.
Therefore, in this case, any receiver writes the received vectors as the rows of a matrix $G'$. Since no errors or rank losses occurred during the communication, this matrix is a generator matrix of a Reed–Solomon code in $\mathcal{S}$. To determine which one it is, it is sufficient to apply the Sidelnikov–Shestakov attack to the generator matrix $G'$, thereby recovering a representative of $\aaa_i$.


\begin{algorithm}[t]
\caption{Recovery of $\aaa_i$ from a received basis of $\mathsf{RS}_{\ell,k}(\aaa_i)$ (no-error case)}
\label{alg:recover_ai_no_error}
\begin{algorithmic}[1]
\Require A set $\mathcal{S}=\{\mathsf{RS}_{\ell,k}(\aaa_1),\dots,\mathsf{RS}_{\ell,k}(\aaa_s)\}$ of RS codes, and
         a collection of received vectors $\{\,{g}'_1,\dots,{g}'_k\,\}\subseteq\mathbb{F}_q^\ell$ forming a basis
         of the transmitted codeword space.
\Ensure The corresponding representative $\ell$-tuple $\aaa_i$ such that the transmitted code is $\mathsf{RS}_{\ell,k}(\aaa_i)$.

\State Form the matrix $G' \in \mathbb{F}_q^{k\times \ell}$ whose rows are the received vectors:
       \[
         G' =
         \begin{pmatrix}
         {g}'_1\\ \vdots\\ {g}'_k
         \end{pmatrix}.
       \]
\Comment{$G'$ is a generator matrix of $\mathsf{RS}_{\ell,k}(\aaa_i)$ in the no-error case}

\State Apply the Sidelnikov--Shestakov attack to $G'$ to recover $\aaa_i$

\State \Return $\aaa_i$.

\end{algorithmic}
\end{algorithm}

\subsection{Error case}


As already mentioned in the previous case, the sender of the network will send a basis $g_1,\ldots,g_k$ of $\mathrm{RS}_{\ell,k}(\aaa_i) \in \mathcal{S}$, for some $i \in [s]$. As some error can occur during the transmission over the network, a receiver may receive 
\[ A G+E, \]
where $A \in \mathrm{GL}(k,q)$ and $E$ is the error matrix.
If $E$ is the zero matrix, then we are in the previous case.

In order to be able to decode quickly, we assume the following:
\begin{center}
    \emph{
    Each receiver receives two vectors $v_1$ and $v_2$ such that $\mathrm{RS}_{\ell,3}(\aaa_i)=\la \underline{1},v_1,v_2 \ra_{\F_q}$.}
\end{center}

In order to make the decoding procedure efficient, we make the following assumption on the received space. We assume that among the received vectors there exist two vectors \(r_a\) and \(r_b\), with \(a\neq b\), such that

$$
\left\langle \underline{1},r_a,r_b\right\rangle_{\F_q}
=
\RS_{\ell,3}(\aaa_i).
$$

In other words, although errors may affect some of the received packets, we assume that the receiver has access to two vectors which, together with the common vector \(\underline{1}\), generate the three-dimensional Reed--Solomon subcode associated with the same evaluation vector \(\aaa_i\) as the transmitted code.

This assumption can be interpreted in terms of the error model \(AG+E\) as requiring that two of the received rows are error-free and that the corresponding linear combinations of the transmitted basis span, together with \(\underline{1}\), the subcode \(\RS_{\ell,3}(\aaa_i)\). Notice that merely requiring two rows of \(E\) to be zero is not sufficient: the corresponding rows must also generate the appropriate two-dimensional complement of \(\langle\underline{1}\rangle_{\F_q}\) inside \(\RS_{\ell,3}(\aaa_i)\).

Our decoding algorithm is based on the fact that it is easy to distinguish a RS code from a random code. This is especially underlined by the behavior of RS codes with respect to the Schur product.

\begin{definition}
Let $C_1$ and $C_2$ be two linear Hamming-metric code in $\Fq^\ell$.
The \textbf{Schur product}  of two linear codes 
$C_1, C_2 \subseteq \mathbb{F}_q^\ell$ is defined as
\[
C_1 * C_2 = \langle c_1 \circ c_2 : c_1 \in C_1,\; c_2 \in C_2 \rangle_{\F_q},
\]
where \(\circ\) is the coordinate-wise product
\[
(x_1,\dots,x_n) \circ (y_1,\dots,y_n) = (x_1 y_1,\dots,x_n y_n).
\]
If $C_1=C_2=C$, we denote $C^2$ the Schur product of a linear $C$ with itself and we call it the \textbf{square code}.
\end{definition}

Studying this product is useful because it reveals structural properties of codes: random codes and highly structured codes (such as GRS codes) behave very differently under the Schur product. This makes it a powerful tool in code-based cryptography and in distinguishing code families.

We list some properties of the Schur product which will be useful later.

\begin{proposition}(see \cite[Propositions 4 and 5]{couvreur2014distinguisher})
    Let $C_1$ and $C_2$ be two linear Hamming-metric codes in $\F_q^\ell$ of dimension $k_1$ and $k_2$ respectively, then:
    \begin{itemize}
        \item $\dim(C_1 * C_2)\leq \dim(C_1)\dim(C_2)$;
        \item $\dim (C_1^2)\leq {k_1+1 \choose 2}$;
        \item The complexity of the computation of a basis of $C_1^2$ is $O(k_1^2\ell^2)$ operations in $\F_q$.
    \end{itemize}
\end{proposition}

Generalized RS codes, and so RS codes, have a relatively small square code. Indeed, as it has been observed in the cryptanalytic setting described in these papers \cite{couvreur2014distinguisher,marquez2013non,wieschebrink2010cryptanalysis}, the following holds.

\begin{proposition}(see e.g. \cite[Proposition 6]{couvreur2014distinguisher})\label{prop:squareRS}
    For $k \leq (\ell+1)/2$,  we have
    \[ \mathrm{RS}_{\ell, k}(\aaa)^2=\mathrm{RS}_{\ell, 2k-1}(\aaa). \]
\end{proposition}

Actually, the previous result holds for generalized RS, but we formulate it directly in the form we will use.

We have all the tools to write the proposed decoding algorithm, see Algorithm~\ref{alg:decoding}. 

\begin{algorithm}
\caption{List decoding algorithm for the subspace code $\mathcal{S}$}
\label{alg:decoding}
\begin{algorithmic}
\Require A set $\mathcal{S}=\{\mathsf{RS}_{\ell,k}(\aaa_1),\dots,
\mathsf{RS}_{\ell,k}(\aaa_s)\}$ of RS codes, and
         a collection of received vectors
         $\{r_1,\dots,r_k\}\subseteq\mathbb{F}_q^\ell$.
\Ensure A list of $\ell$-tuples containing $\aaa_i$ such that the transmitted
code is $\mathsf{RS}_{\ell,k}(\aaa_i)$.

\State $\mathcal{L}\gets\emptyset$

\State Compute the square product $C^2$ of
$C=\la r_1,\ldots,r_k\ra_{\F_q}$

\If{$\dim(C^2)=2k-1$}
    \State Use the Sidelnikov--Shestakov attack on $C$ to recover
    $\aaa$ from $C=\mathrm{RS}_{\ell,k}(\aaa)$

    \If{the attack succeeded and $\aaa$ is equivalent to one of
    $\{\aaa_1,\ldots,\aaa_s\}$}
        \State {$\mathcal{L}\gets
        \mathcal{L}\cup\{\aaa\}$}
    \EndIf
\EndIf

\For{$1\leq i<j\leq k$}
    \State Compute
    $C'=\la \underline{1},r_i,r_j\ra_{\F_q}$

    \If{$\dim(C')=3$ and $\dim((C')^2)=5$}
        \State Use the Sidelnikov--Shestakov attack on $C'$ to recover
        $\aaa$ from
        $C'=\mathrm{RS}_{\ell,3}(\aaa)$

        \If{the attack succeeded and $\aaa$ is equivalent to one of
        $\{\aaa_1,\ldots,\aaa_s\}$}
            \State $\mathcal{L}\gets
            \mathcal{L}\cup\{\aaa\}$
        \EndIf
    \EndIf
\EndFor

\State \Return $\mathcal{L}$

\end{algorithmic}
\end{algorithm}



Algorithm \ref{alg:decoding} can give more than one output. This is due to the fact that in a RS codes there might be different three-dimensional RS codes.
For instance, if the receiver receives 
\[ r_1=\underline{1},r_2=(g_1,\ldots,g_\ell), r_3=(g_1^2,\ldots,g_\ell^2), r_4=(g_1^3,\ldots,g_\ell^3), r_5=(g_1^4,\ldots,g_\ell^4), \ldots
\]
Algorithm \ref{alg:decoding} will output both $(g_1,\ldots,g_\ell)$ and $(g_1^2,\ldots,g_\ell^2)$, as both $\mathsf{RS}_{\ell,3}((g_1,\ldots,g_\ell))$ and $\mathsf{RS}_{\ell,3}((g_1^2,\ldots,g_\ell^2))$ are contained in $\mathsf{RS}_{\ell,k}((g_1,\ldots,g_\ell))$.
Hence, the success of the decoding algorithm is related to how often it happens that given two vectors $r_i,r_j \in \Fq^\ell$ the code $C'=\la \underline{1},r_i,r_j\ra_{\F_q}$ is a three-dimensional RS code.
If this happens \emph{few} times the decoding algorithm will produce a small list of evaluation points that need to be checked in order to find the right one.

To this aim, we need to study the Schur square of a random code of dimension three. Known results, see e.g. \cite{faugere2013distinguisher}, assert that a random code of dimension $k$ has Schur-square dimension equal to ${k+1 \choose 2}$ with very high probability for large values of $k$.
We need to control this condition in the case of dimension-three codes.

\begin{remark}
Algorithm \ref{alg:decoding} runs in polynomial time with respect to $\ell$.
\end{remark}

\subsubsection{The square code of a three-dimensional random code}

In this subsection, we bound the probability that the Schur square of a random three-dimensional code has dimension less than $6$, and show that this probability decreases as $q$ grows.
These estimates can be viewed as a refinement, in our special case, of the results in \cite{cascudo2015squares}.

Let us fix a linear Hamming-metric code $C \subseteq \Fq^\ell$ of dimension $3$ and consider an its generator matrix $G$.
Denote by $\pi_1,\ldots,\pi_\ell \in \Fq^3$ the columns of $G$.
Define the following map
\[ \text{ev}_C \colon Q \in \text{Quad}(\Fq^3) \mapsto (Q(\pi_1),\ldots,Q(\pi_\ell))\in \Fq^\ell, \]
where $\text{Quad}(\Fq^3)$ is the $\Fq$-vector space of quadratic forms over $\Fq$ in three variables.
Observe that $\text{ev}_C$ is an $\Fq$-linear map,
\[ \mathrm{Im}(\textrm{ev}_C)=C^{(2)}\,\,\text{and}\,\, \dim(\text{Quad}(\Fq^3))=6. \]

Therefore, we have the following.

\begin{proposition}
    Let $C \subseteq \Fq^\ell$ be a linear Hamming-metric code of dimension $3$. 
    Then $\dim(C^{(2)})=6$ if and only if $\dim(\ker(\text{ev}_C))=0$.
\end{proposition}

We now estimate the probability that the Schur square of a random three-dimensional code has dimension less than $6$.

\begin{theorem}
Let $C\subseteq \mathbb{F}_q^\ell$ be a random three-dimensional code. Then 
\[
\Pr\bigl(\dim(C^{(2)})<6\bigr)
\leq
\frac{q^6-1}{q-1}
\left(\frac{2}{q}\right)^\ell (1 + o(1)).
\]
In particular, for fixed $\ell$,
\[
\Pr\bigl(\dim(C^{(2)})<6\bigr)
\leq
2^\ell q^{5-\ell}(1+o(1)).
\]
\end{theorem}

\begin{proof}
Let $G$ be a generator matrix of $C$, and denote its columns by $\pi_1,\ldots,\pi_\ell\in\mathbb{F}_q^3$.
Recall that $C^{(2)}$ is the image of the evaluation map
\[
\operatorname{ev}_C:\operatorname{Quad}(\mathbb{F}_q^3)\longrightarrow
\mathbb{F}_q^\ell,
\qquad
Q\longmapsto (Q(\pi_1),\ldots,Q(\pi_\ell)).
\]
Hence $\dim(C^{(2)})<6$ if and only if there exists a nonzero quadratic form
$Q\in\operatorname{Quad}(\mathbb{F}_q^3)$ such that $Q(\pi_i)=0$ for every $i\in [\ell]$.

Fix a nonzero quadratic form $Q$. By the Schwartz--Zippel bound, or simply by
the fact that a nonzero polynomial of degree $2$ in three variables has at most
$2q^2$ zeros in $\mathbb{F}_q^3$, we have
\[
\Pr(Q(X)=0)
=
\frac{|\{X\in\mathbb{F}_q^3:Q(X)=0\}|}{q^3}
\leq
\frac{2q^2}{q^3}
=
\frac{2}{q}.
\]
Therefore, assuming the columns are chosen independently,

\begin{align*}
\Pr(Q(\pi_i)=0,\; \forall i\in [\ell] \mid  \dim(\textup{Span}(\pi_1, \ldots,\pi_{\ell}))=3)
&\leq \frac{\Pr(Q(\pi_i)=0,\; \forall i\in [\ell])}{\Pr(\dim(\textup{Span}(\pi_1, \ldots,\pi_{\ell}))=3))}
 \\
&\leq \frac{\left(\frac{2}{q}\right)^\ell}{\prod_{j=0}^2 (1 - q^{j-\ell})}\\
& = \left( \frac{2}{q} \right)^{\ell} (1 + o(1))\;.
\end{align*}

Now we take a union bound over all nonzero quadratic forms, up to scalar
multiplication. Since the size of quadratic forms over $\F_q^3$, up to a scalar multiplication, are $\frac{q^6-1}{q-1}$, we obtain
\[
\Pr\bigl(\dim(C^{(2)})<6\bigr)
\leq
\frac{q^6-1}{q-1}
\left(\frac{2}{q}\right)^\ell (1 + o(1)).
\]
This proves the claimed bound.
\end{proof}

Therefore, we can estimate the probability that a random code is a RS code.

\begin{corollary} \label{cor-prob}
    The probability that a random code $C$ of dimension 3 in $\Fq^\ell$ is an RS code is at most $$\frac{q^6-1}{q-1}
\left(\frac{2}{q}\right)^\ell (1 + o(1)).$$
\end{corollary}
\begin{proof}
    If $C$ is a RS code of dimension three, then $\dim(C^{(2)})=5$ by Proposition \ref{prop:squareRS}. Hence, the probability that $C$ is a RS code is upper bounded by the probability that the code $C$ has Schur-square dimension strictly less than $6$.
    The assertion now follows from the previous theorem.
\end{proof}

Thus, a random three-dimensional candidate is unlikely to pass the Schur-square test when $q$ is large and $\ell \geq 6$ is fixed. This supports the expectation that Algorithm \ref{alg:decoding} typically produces a short list, although a rigorous list-size bound for the actual candidates tested by the algorithm would require a separate analysis.

With the counting of RS codes we can immediately compute the probability that a random $3$-dimensional subspace of $\mathbb{F}_q^\ell$, with $\ell \geq 3$, is a RS code, by dividing by the $q$-binomial coefficient $\qbin{\ell}{3}{q}$, the total number of three-dimensional subspaces of $\mathbb{F}_q^\ell$.

\begin{corollary}
The probability that a random code $C$ of dimension $3$ in $\mathbb{F}_q^\ell$ is a RS code is 

$$\frac{(q-2)!}{(q-\ell)!\qbin{\ell}{3}{q}} \sim \bigg(\frac{1}{q}\bigg)^{2\ell-7},$$
where $\qbin{\ell}{3}{q}=\frac{(q^\ell-1)(q^\ell-q) \ldots (q^\ell-q^2)}{(q^3-1)(q^3-q)\ldots(q^3-q^2)}$.
\end{corollary}

\section{Conclusion and open questions}
In this paper we introduced RS sunflowers, namely subspace sunflowers whose petals are RS codes. We focused mainly on the natural one-dimensional center generated by the all-one vector and translated the sunflower condition into an explicit algebraic rank condition involving generalized $V$-matrices. This gives a concrete criterion for constructing families of RS codes with prescribed pairwise intersection and, equivalently, constant-dimension subspace codes with minimum distance $2k-2$.

We obtained a complete description in the case $k=2$, where any two distinct RS codes intersect exactly in the center, and the sunflower size is determined by counting RS codes up to affine equivalence of the evaluation vectors. For larger dimensions, we gave an explicit recursive construction based on the generalized $V$-matrix criterion and derived lower bounds on the number of petals it produces. We also gave greedy existence bounds for MDS code families and RS sunflowers. For fixed $k\geq3$ and $\ell\geq2k-1$, these arguments give exponents $2(\ell-2k+2)$ and $\ell-2k+2$, respectively. The larger MDS bound concerns families whose pairwise intersections have dimension at most one, while the RS families have the common center $\langle\underline{1}\rangle_{\Fq}$.

Finally, we discussed an application to random linear network coding. The behavior of RS codes under the Schur product provides a way to recognize candidate petals and leads to a list-recovery procedure. 

We conclude with several directions for future research.
\begin{enumerate}
    \item \textbf{Extremal size of RS sunflowers.} Determine, or sharply estimate, the maximum possible size of an $(\ell,k,q;1)$ RS sunflower for $k\geq 3$. In particular, can the greedy lower bound $\Omega_{k,\ell}(q^{\ell-2k+2})$ be improved, or can a matching upper bound be proved?

    \item \textbf{Sharper explicit constructions.} Improve the recursive construction by finding larger sets of evaluation vectors satisfying the generalized $V$-matrix rank condition. One goal is to match the exponent of the greedy RS bound within the recursive construction.

    \item \textbf{Other centers and higher-dimensional intersections.} Extend the theory beyond the center $\langle\underline{1}\rangle_{\Fq}$. Natural questions include the construction and classification of RS sunflowers with higher-dimensional centers, as well as families whose intersections contain prescribed RS subcodes.

    \item \textbf{Further applications.} Random linear network coding gives a natural first application of RS sunflowers, but their combination of large distance, explicit algebraic structure, and recognizability may be useful in other settings as well. It would be interesting to identify further coding-theoretic or combinatorial applications where this additional structure can be exploited.
\end{enumerate}

\section*{Acknowledgments}
The last author is very grateful for the hospitality of the Algebra group at DTU, he was visiting DTU during the development of this research in August 2025.
This work was supported by a research grant (VIL”52303”) from Villum Fonden.
The research of the last author was partially supported by the Italian National Group for Algebraic and Geometric Structures and their Applications (GNSAGA - INdAM).

\newpage
\bibliographystyle{abbrv}
\bibliography{biblio.bib}

\end{document}